\documentclass[11pt]{amsart}
\usepackage{amsmath,amssymb,xypic,array}
\usepackage[T1]{fontenc}
\usepackage{amsfonts}
\usepackage{amsmath}
\usepackage{amssymb}
\usepackage{amsthm}
\usepackage{amssymb}
\usepackage{setspace}
\usepackage{tikz}
\usepackage{booktabs}
\usepackage{hyperref}
\usepackage{longtable}
\usepackage{geometry}
\usetikzlibrary{trees}
\usepackage{amsfonts}
\usepackage{amsmath}

\usepackage[english]{babel}
\usepackage[latin1]{inputenc}
\usepackage{amsmath,amssymb,amsthm,epsfig}
\usepackage{color}
\usepackage{mathrsfs}
\usepackage{latexsym}
\usepackage{amsthm}
\usepackage{tikz}
\usepackage{textcomp}
\usetikzlibrary{trees}

\providecommand{\U}[1]{\protect\rule{.1in}{.1in}}
\providecommand{\U}[1]{\protect\rule{.1in}{.1in}}
\providecommand{\U}[1]{\protect\rule{.1in}{.1in}}
\providecommand{\U}[1]{\protect\rule{.1in}{.1in}}
\providecommand{\U}[1]{\protect\rule{.1in}{.1in}}
\input{xy}
\xyoption{all}
\theoremstyle{theorem}

\newtheorem{definition}{Definition}
\newtheorem{proposition}{Proposition}
\newtheorem{corollary}{Corollary}
\newtheorem{thm}{Theorem}
\newtheorem{lemma}{Lemma}

\newtheorem{rmk}{Remark}
\newtheorem{example}{Example}
\newtheorem{RCF}{Rationality Conjecture for Flag}
\newcommand{\F}{\mathcal{F}}
\newcommand{\bc}{\begin{center}}
\newcommand{\ec}{\end{center}}
\newcommand{\noi}{\noindent}

\numberwithin{equation}{section}

\newcommand{\codim}{\operatorname{codim}}

\usepackage{ifthen,latexsym}
\usepackage{times, amscd}

\newcommand{\R}{\Bbb{R}}

\newcommand{\G}{\mathcal{G}}

\begin{document}
\title{Residues for flags of holomorphic foliations}

\author[Jean-Paul Brasselet]{Jean-Paul Brasselet}
\address{\sc Jean-Paul Brasselet\\
I2M-CNRS\\
Luminy Case 930\\
F-13288 Marseille Cedex 9
\\ France}
\email{jean-paul.brasselet@univ-amu.fr}

\author[Maur\'icio Corr\^ea]{Maur\'icio Corr\^ea}
\address{\sc Maur\'icio Corr\^ea\\
UFMG\\
Avenida Ant\^onio Carlos, 6627\\
30161-970 Belo Horizonte\\ Brazil}
\email{mauricio@mat.ufmg.br}

\author[Fernando Louren\c{c}o]{Fernando Louren\c{c}o}
\address{\sc Fernando Louren\c{c}o\\
UFLA\\
Av. Doutor Sylvio Menicucci, 1001\\
Kennedy
37200000 - Lavras \\Brasil \\
}
\email{fernando.lourenco@dex.ufla.br}

\date{\today}
\subjclass[2010]{Primary 57R30; Secondary  14J45, 57R32, 53C12}
\keywords{Residues, Holomorphic foliations, Characteristic classes.}

\begin{abstract}
 In this work we prove a   Baum-Bott type residue theorem for flags of holomorphic foliations.
 We prove some relations between the residues of the  flag and the residues of their correspondent foliations.  
 We define the  Nash residue for flags and  we give a partial answer to the  Baum-Bott  type rationality conjecture in this context .
\end{abstract}

\maketitle

\setcounter{tocdepth}{1}

\tableofcontents
\section*{Introduction}

A $2$-flag of foliations is a pair of foliations $(\F_1,\F_2)$
such  that the leaves of $\F_1$ are contained in the leaves  of
$\F_2$. We call $\F_1$  a subfoliation of $\F_2$. In this work, we
considerer  $2$-flags  formed by $2$ holomorphic foliations on a
complex manifold $M$ of dimension $n$. Our main result is a
Baum-Bott type residue theorem. 
We believe that the study of characteristic
classes of singular flags  of holomorphic foliations  can  be useful to give information about these structures.
The study of characteristic classes
of real foliations was firstly considered by   Feigin [\ref{Fei}],
where he proposed two constructions for characteristic classes of
flags in an attempt to answer a question about   topological
obstruction for integrability. Several other authors studied
characteristic classes of flags,   \cite{CorMa,Domin}.

In the holomorphic context, flags of holomorphic foliations appear
naturally on works about foliations with algebraic and rationally
connected leaves: Miyaoka \cite{Miy}, Bogomolov-Mcquillan \cite{BM},
and Kebekus-Sola Conde-Toma \cite{KST}. Other very important
situation where appear flags of foliations is the so called
Brunella's conjecture:
\\
 \textit{If $\F$ is a codimension one holomorphic foliation on $\mathbb{P}^{3}$, then $\F$ either  admits an algebraic invariant surface or admits a  subfoliation by algebraic curves.}

In the case where $\F$ admits a  subfoliation by algebraic curves we
have a $2$-flag. Then the study of the singularities of these
foliations is important.  Some authors studied
flags of singular holomorphic foliations: in [\ref{CoSo}] Corr\^ea
and Soares  proved an inequality involving the degrees of two
distributions (foliations) which form a flag on projective spaces.
 Mol in [\ref{Mol}] studied the polar classes  of    flags of holomorphic
 foliations.
\\
\\
\paragraph{\bf Acknowledgments.}
We are grateful to Maria Aparecida Soares  Ruas  and Marcio Soares for
interesting and fruitful conversations.
The first named author thanks partial support by FAPESP Process: 2015/06697-9 and BREUDS European programme.
 The second named author thanks partial support by 
CAPES, CNPq, FAPEMIG and FAPESP grant 2015/20841-5.

\section{Baum-Bott residues theory for flags}

The main results of this section are a  Bott vanishing theorem and a
Baum-Bott theorem for flags, where the first one is  "finer" than 
the classical result of Bott. The reason is that  our  theorem
detects some characteristic classes of flags that the Bott vanishing
theorem for foliations does not detect. These classes are legitimate
of flags, see remark \ref{2.17}. Here we apply the localization
theory of characteristic classes developed mainly by   Lehmann and
 Suwa to localize the characteristic classes. For details
on  this theory we refer to  [\ref{Suwa}].

\subsection{Flags of holomorphic foliations}

Let $M$ be a connected complex manifold of dimension $n$. Let us
denote by $\Theta_{M}$ the tangent sheaf of $M$.

\begin{definition} A singular holomorphic foliation $\mathcal{F}$ of
dimension $q$ on $M$ is a coherent subsheaf of rank $q$ of
$\Theta_{M}$ that is involutive. \noi Here involutive means that the
stalks of $\F$ are invariant by   Lie bracket. We denote by $k =
n-q$ the codimension of the foliation $\mathcal{F}.$

\end{definition}

Let us denote by $\mathcal{N}_{\mathcal{F}}$ the quotient sheaf
$\Theta_{M} / \mathcal{F}$, so that we have the exact sequence.
$$ 0 \longrightarrow \mathcal{F} \longrightarrow \Theta_{M}
\longrightarrow \mathcal{N}_{\mathcal{F}} \longrightarrow 0. $$

We define the singular set $S(\mathcal{F})$ of the foliation $
\mathcal{F} $ to be the singular set
 S$(\mathcal{N}_{\mathcal{F}})$ of the normal sheaf of the
foliation. We will assume that $\codim{S(\mathcal{F})}\geq 2. $ A
regular foliation $\F$ is a foliation such that S($\F)$ is empty.

Let $\F\subsetneq \Theta_{M}$ a distribution  of dimension  $r$ on
$M$.  The sheaf $N^*_\F:=(\Theta_{M}/\F)^*$ is a saturated subsheaf
of $\Theta_{M}^*=\Omega^1_M$ of rank equal to $\codim(\F)=k$. The $k$-th
wedge product of the inclusion $N^*_\F\subset \Omega^1_M$ gives rise
to a nonzero twisted holomorphic  $k$-form $\omega\in H^{0}(M,
\Omega_{M}^{k}\otimes \det(N_{\F}))$.

\begin{example}
Let $\mathbb{P}^{n}$ be the complex projective space of dimension
$n$. A holomorphic foliation $\mathcal{F}$ on $\mathbb{P}^{n}$, of
codimension $k$,    is given by an integrable twisted $k$-form
$$
\omega \in H^{0}(\mathbb{P}^{n}, \Omega_{\mathbb{P}^{n}}^{k}\otimes
\mathcal{N}),
$$
 where $\mathcal{N} = \det \mathcal{N}_{\mathcal{F}}$. Moreover, we can define the degree of $\mathcal{F}$, denoted by $\mbox{deg}(\mathcal{F}) \in \mathbb{Z}$,
 as the degree of the zero locus of $\iota^{\ast} \omega$, where $\iota : \mathbb{P}^{k} \longrightarrow \mathbb{P}^{n}$
 is a linear embedding of a generic $k$-plane. Since $\Omega_{\mathbb{P}^{n}}^{k} = \mathcal{O}_{\mathbb{P}^{k}}(-k-1)$ it follows at once that $\mathcal{N} = \mathcal{O}_{\mathbb{P}^{n}}(deg(\mathcal{F}) + k +1).$
\end{example}

\begin{definition} Let $\mathcal{F}_{1},...,\mathcal{F}_{t}$ be $t$
holomorphic foliations on $M$ of dimensions $q = (q_{1},...,q_{t})$.
We say that $\mathcal{F} := (\mathcal{F}_{1},...,\mathcal{F}_{t})$
is a flag of holomorphic foliations if for each $i = 1,...,t-1, \ \ \mathcal{F}_{i}$ is a coherent
sub $\mathcal{O}_{M}$-module of $\mathcal{F}_{i+1}.$

\end{definition}

In the above definition, we say that $\mathcal{F}_{i}$ leaves
$\mathcal{F}_{j} (i < j) $ invariant for each $i = 1,\dots ,t-1$.
Note that, for $ x \in M \setminus \cup_{i = 1}^{t} S
(\mathcal{F}_{i})$ the inclusion relation $ \mathcal{F}_{x,1}
\subset \dots \subset
 \mathcal{F}_{x,t}$ holds, giving that the leaves of
$\mathcal{F}_{i}$ are contained in leaves of $\mathcal{F}_{j}$. When
$t = 2$, we have a diagram of   exact sequences of sheaves:
$$ \xymatrix{ 0 \ar[rd]  &  & 0 \ar[ld] &  & 0  \\
   & \mathcal{F}_{1} \ar[rd] \ar[dd] &  & \mathcal{N}_{2} \ar[lu] \ar[ru] &  \\
   &  & \Theta_{M} \ar[ru] \ar[rd]  &  &  \\
   & \mathcal{F}_{2} \ar[ru]   \ar[rd] &  & \mathcal{N}_{1} \ar[uu]  \ar[rd] &  \\
 0 \ar[ru] &  & \mathcal{N}_{1,2} \ar[ru] \ar[rd]  &  & 0 \\
  & 0 \ar[ru] &   &  0 &  }$$
This diagram appears in the studies of Feigin in real case, see
[\ref{CorMa}, pg 64]. We define the singular set $S(\mathcal{F})$ of
the flag $\mathcal{F}$ to be the singular set $S(\mathcal{F}_{1})
\cup ... \cup S(\mathcal{F}_{t})$ and $\mathcal{N}_{\mathcal{F}} =
\mathcal{N}_{1,2} \oplus ... \oplus \mathcal{N}_{t-1, t} \oplus
\mathcal{N}_{t}$ be the normal sheaf of the flag, where
$\mathcal{N}_{i, j}$ is the quotient sheaf $ \mathcal{F}_{i} /
\mathcal{F}_{j} (i < j)$.

\begin{example} Let  $\pi : M \longrightarrow Y $ a surjective holomorphic map,
where $M$ and $Y$ are complex manifolds. Given a regular holomorphic
foliation $\mathcal{G}$ of codimension one on $Y$ one has that
$\mathcal{F}_{2} := \pi^{\ast} \mathcal{G}$ is a codimension one
foliation on $M$. We denote by $\mathcal{F}_{1}$ the foliation
induced by the  fibration $\pi$. Observe that
$S(\mathcal{F}_{1})$ is the singular set of the map $\pi.$ Then
$\mathcal{F} = (\mathcal{F}_{1}, \mathcal{F}_{2} )$ is a flag on $M$
such that $S(\mathcal{F}_{1}) = S(\mathcal{F}_{2})$.

\end{example}

\begin{example} A meromorphic map $\varphi : M \dashrightarrow Y$, where $M$ and $Y$ are complex manifolds, is a first integral of a foliation $\mathcal{F}$ on $M$ if the leaves of $\mathcal{F}$
are contained in the fibers of $\varphi$. In this situation,
$\mathcal{F}$ is a subfoliation of the singular foliation induced by
$\varphi$.

\end{example}

\subsection{Chern-Weil theory of characteristic classes}

We will start this subsection by a review on Chern-Weil theory of
characteristic classes of vector bundles.

We denote by
$A^p(M)$ the complex vector space of complex valued $C^{\infty}$ $p$-forms on $M$. Also,
we let $A^p(M,E)$ be the vector space of  $C^{\infty}$ sections of the bundle
$\bigwedge^p{(T_{\R}^cM)}^*\otimes E$ on $M$, where ${(T_{\R}^cM)}^*$ denotes
the dual of the complexification of the real tangent bundle $T_{\R}M$ of $M$.
Finally, we denote by  $TM$  the  holomorphic tangent
bundle of $M$.

\begin{definition} A connection for a complex vector bundle $E$ on $M$ is
a $\mathbb{C}$-linear map
$$ \nabla: A^{0}(M,E) \longrightarrow A^{1}(M,E)$$
such that
$$ \nabla(fs) = df \otimes s + f \nabla (s) \ \ \ \mbox{for} \ \ \ f \in A^{0}(M) \ \ \
\mbox{and} \ \ \ s \in A^{0}(M,E).$$

\end{definition}
If $H$ is a subbundle of the complexified tangent bundle $T^{c}M$,
then its dual $H^{\ast}$ is canonically viewed as a quotient of
$(T_{\R}^cM)^{\ast}$. We denote by $\rho$ the canonical projection $
(T_{\R}^cM)^{\ast} \longrightarrow H^{\ast}$.

\begin{definition} A partial connection for $E$ is a pair $(H,\delta)$ of
a subbundle $H$ of $T_{\R}^cM$ and a $\mathbb{C}$-linear map
$$ \delta : A^{0}(M,E) \longrightarrow A^{0}(M,H^{\ast}\otimes E)$$
such that
$$ \delta(f s) = \rho (df) \otimes s + f \delta (s) \ \ \ \mbox{for} \
\ \ f \in A^{0}(M) \ \ \ and \ \ \ s \in A^{0}(M,E).$$
\end{definition}

\begin{definition} Let $(H,\delta)$ be a partial connection for $E$. We say that a connection $\nabla$ for $E$ extends $(H,\delta)$ if the following  diagram is commutative
$$ \xymatrix{ A^{0}(M,E) \ar[r]^{\nabla} \ar[rd]_{\delta}  &  A^{0}(M,(T_{\R}^cM)^{\ast}\otimes E) \ar[d]^{\rho \otimes Id} \\  & A^{0}(M,H^{\ast}\otimes E) } $$

\end{definition}

\begin{lemma} Every  partial connection for $E$ admits  a connection that extends it.

\end{lemma}

An important class of partial connections comes from "action" of
involutive subbundles of the  tangent bundles of manifolds:

\begin{definition} Let $F \subset TM$ be an involutive subbundle of
$TM$. An action of $F$ on vector bundle $E$ is a
$\mathbb{C}$-bilinear map:
$$ \alpha : A^{0}(M,E) \times A^{0}(M,F) \longrightarrow
A^{0}(M,E)$$
 satisfying the following conditions for $f \in A^{0}(M), \ \ u, v
\in A^{0}(M,F) \ \ s \in A^{0}(M,E)$: \\\

\noi 1) $\alpha([u,v],s) = \alpha(u, \alpha(v,s)) -
\alpha(v,\alpha(u,s))$ ;   \\

\noi 2) $\alpha(f.u,s) = f.\alpha(u,s)$ ;  \\

\noi 3) $\alpha(u,f.s) = u(f).s + f\alpha(u,s)$ ;  \\

\noi 4) $\alpha(u,s)$ is holomorphic whenever $u$ and $s$ are. \\
\\
If $E$ admits an action of $F$ on says that $E $ is an $F$-bundle.

\end{definition}

\begin{lemma} Let $\alpha$ be an action of $F$ on $E$ and let
$$ \delta_{\alpha} : A^{0}(M,E) \longrightarrow A^{0}(M,F^{\ast} \otimes
E) \simeq A^{0}(M, Hom(F,E)) $$
be defined by $\delta_{\alpha}(s,u) = \alpha(u,s).$ Then the pair $(F,\delta_{\alpha})$ is a partial connection for $E$.

\end{lemma}

\begin{definition} Let $\alpha$ be an action of $F$ on $E$. An
$F$-connection for $E$ is a connection which extends the partial connection $(F\oplus \overline{TM}, \delta_{\alpha} \oplus
\overline{\partial} )$.

\end{definition}

\subsection{Bott vanishing theorem for $2$-flags}

Now, we will use the Chern-Weil theory of characteristic classes, in
order to describe the Bott vanishing theorem for flags. This is a
holomorphic version of the vanishing theorem due to Cordero-Masa,
[\ref{CorMa}, Theorem 3.9, pg 71].

\begin{thm}\label{2.15} Let $\mathcal{F} = (\mathcal{F}_{1}, \mathcal{F}_{2} )$ be a 2-flag of holomorphic foliations on a compact complex manifold $M$ of dimension $n$ such that $\mathcal{F}_{1} =
\mathcal{O}(F_1)$ and $\mathcal{F}_{2} = \mathcal{O}(F_2)$ with $F_{1} \subset F_{2}
\subset TM$ involutive subbundles.
Let 
$E = ( E_{1} , E_{2})$ be a pair of  vector bundles on $M$ with $E_{1}$ an
$F_{1}$-bundle, $E_{2}$ an $F_{2}$-bundle.  Let $\varphi_{1}$ and $\varphi_{2}$
be homogeneous symmetric polynomials, of degrees $d_{1}$ and
$d_{2}$, such that at least one of the inequalities
\begin{equation}\label{eq. 1}
d_{1} > n - \mbox{rank} F_{1} \ \ \ or \ \ \ d_{2} > n - \mbox{rank} F_{2}
\ \ \ or \ \ \ d_{1} + d_{2}
> n - \mbox{rank} F_{1}
\end{equation}
\noi is satisfied. Then the form  $\varphi_{1}(E_{1}) \smile \varphi_{2}(E_{2}) =0$ vanishes. 

\end{thm}

\begin{proof} \noi Let us denote rank($F_{1}) = p_{1}$, rank$(F_{2}) = p_{2}$,
rank$(E_{1}) = r_{1}$ and rank$(E_{2}) = r_{2}$.
\\

Let $ \alpha_{i} : A^{\circ}(M, F_{i}) \times A^{\circ}(M, E_{i})
\rightarrow A^{\circ}(M, E_{i})$ be an action of $F_{i}$
 on  $E_{i}$, for $i = 1, 2.$  and let  $\bigtriangledown^{i}$ be an $F_{i}$-connection for
$E_{i}$. Let us consider  a coordinate neighborhood $ \{ U,
(z_{1},...,z_{n}) \}$ on   $M$ such that $F_{1}$ and $F_{2}$, can be
written (spanned) by:

\noi $F_{1} = <v_{1},...,v_{p_{1}}>$ \ \ \ and \ \ \ $F_{2} =
<v_{1},...,v_{p_{1}},...,v_{p_{2}}>$, where $\displaystyle v_{i} =
\frac{\partial}{\partial z_{i}}$. \\

We have holomorphic frames (see \cite{Suwa}), $S^{1} =
(s_{1}^{1},...,s_{r_{1}}^{1})$ of $E_{1}|_{U}$ and $S^{2} =
(s_{1}^{2},...,s_{r_{2}}^{2})$ of $E_{2}|_{U}$ such that \\

$\bullet \ \ \alpha_{1}(v_{i}, s_{\nu}^{1}) = 0,$ \ \ \ for \ \ \ $i
= 1,...,p_{1}$
and $\nu = 1,...,r_{1}.$ \\

$\bullet \ \ \alpha_{2}(v_{i}, s_{\nu}^{2}) = 0,$ \ \ \ for \ \ \ $i
= 1,...,p_{2}$
and $\nu = 1,...,r_{2}.$ \\

\noi Now let $\Theta^{1} = (\Theta^{1}_{\nu \mu})$ and $\Theta^{2} =
(\Theta^{2}_{\nu \mu})$ be connection matrices  of
$\nabla^{1}$ and $\nabla^{2}$ respectively, i.e \\

$\nabla^{1}(s_{\nu}^{1}) = \sum_{\mu = 1}^{r_{1}} \Theta^{1}_{\nu
\mu} s_{\mu}^{1}$, \hskip 2truecm 
$\nabla^{2}(s_{\nu}^{2}) = \sum_{\mu = 1}^{r_{2}} \Theta^{2}_{\nu
\mu} s_{\mu}^{2}$. \\

\noi We have \\

$\nabla^{1}(s_{\nu}^{1})(v_{i}) = \alpha_{1}(v_{i},s_{\nu}^{1}) = 0$, \hskip 2truecm 
$\nabla^{2}(s_{\nu}^{2})(v_{i}) = \alpha_{2}(v_{i},s_{\nu}^{2}) =
0.$
\\

Then, we have ${\displaystyle   i_{\frac{\partial}{\partial z_{i}}} \Theta^{1}_{\nu \mu} = 0, }$ \  for all \  $i = 1,...,p_{1}$ \  and \  $\nu, \mu =
1,...,r_{1}.$
This shows that each $\displaystyle \Theta^{1}_{\nu \mu}$  is of the form $\sum_{i = p_{1} + 1}^{n}
f_{i}^{\nu \mu} dz_{i}$ with $f_{i}^{\nu \mu} \in \mathcal{O}(U).$ Thus, the curvature matrix has the following property
\begin{center}
 $K^{1} = (K^{1}_{\nu \mu})$ \ \ \ with \ \ \ $K^{1}_{\nu \mu} =
\sum_{i = p_{1}+1}^{n} \eta_{i}^{\nu \mu} dz_{i}$ \ \ \ where \ \ \ $\eta_{i}^{\nu
\mu} \in \Omega^{1}(U)$.
\end{center}

\noi  Similarly $\Theta^{2}_{\nu \mu} = \sum_{i = p_{2} + 1}^{n}
g_{i}^{\nu \mu} dz_{i}$ \ \ and \ \ $K^{2}_{\nu \mu} =
\sum_{i = p_{2}+1}^{n} \omega_{i}^{\nu \mu} dz_{i}$.
Then $$
\varphi_{1}(E_{1}) \smallsmile
\varphi_{2}(E_{2}) = \varphi_{1}(K^{1}) \smallsmile
\varphi_{2}(K^{2}).
$$
Therefore, if either $d_{1} > n - p_{1}$ or $d_{2} > n - p_{2}$ or $d_{1} + d_{2} > n - p_{1}$ then 
$\varphi_{1}(E_{1}) \smallsmile \varphi_{2}(E_{2})=0$.
\end{proof}

Now, since  we have on hand a vanishing theorem for holomorphic flags, we can try to compute characteristic classes associated to the flags. 

Let $\mathcal{F} = (\mathcal{F}_{1}, \mathcal{F}_{2} )$ be a 2-flag of holomorphic foliations on a compact complex manifold $M$. Away from the singular set $S(\mathcal{F})$ of  the flag, $\mathcal{F}_{1}$ and $\mathcal{F}_{2}$
are locally free sheaves, then there exist vector bundles $F_{1}^{0}$ and $F_{2}^{0}$ on 
$M^{0} :=M \setminus S(\mathcal{F})$ such that
$\mathcal{O}(F_{1}^{0}) = \mathcal{F}_{1} $ and $\mathcal{O}(F_{2}^{0}) = \mathcal{F}_{2}.$ We have that $F_{1}^{0}$ and $F_{2}^{0} $ are subbundles for $TM^0$. Also, let $\displaystyle N_{F_{2}^{0}} = TM^{0} / F_{2}^{0}$ and $\displaystyle N_{1 2} = F_{2}^{0} / F_{1}^{0}$, then denote $\mathcal{N}_{2}
:= \mathcal{O}(N_{F_{2}^{0}})$ and $\displaystyle \mathcal{N}_{1 2}
:= \mathcal{F}_{2} / \mathcal{F}_{1} = \mathcal{O}(N_{1 2}).$ Finally, we denote by $\mathcal{N}_{\mathcal{F}}$ the sum $\mathcal{N}_{\mathcal{F}} = \mathcal{N}_{1 2} \oplus \mathcal{N}_{2}$. 

Let us consider $\varphi := (\varphi_{1}, \varphi_{2} )$  where $\varphi_{1}$ and $\varphi_{2}$  are  homogeneous symmetric
polynomials of degree $d_{1}
$ and $d_{2}$ respectively, satisfying (\ref{eq. 1}).  We will show that the class $\varphi(\mathcal{N}_{\mathcal{F}}) = \varphi_{1}(\mathcal{N}_{12}) \smile  \varphi_{2}(\mathcal{N}_{2})$
 is localized at $S(\mathcal{F})$ and we will compute this class. 

\begin{thm}\label{2.16} Let $\mathcal{F} = (\mathcal{F}_{1}, \mathcal{F}_{2} )$ be a 2-flag of holomorphic foliations on a compact complex manifold
$M$ of dimension $n$. Let $\varphi_{1}, \varphi_{2}$ be homogeneous symmetric polynomials , of degree $d_{1}
$ and $d_{2}$, respectively,   satisfying (\ref{eq. 1}). Then for each compact connected component $S_{\lambda}$ of $S(\mathcal{F})$ there exists Res$_{\varphi_{1}, \varphi_{2}} (\mathcal{F} , \mathcal{N}_{\mathcal{F}}, S_{\lambda} ) \in H_{2n - 2(d_{1} + d_{2} )} (S_{\lambda}; \mathbb{C})$ such that, in $H_{2n - 2(d_{1} + d_{2} )} (M; \mathbb{C})$ we have
\begin{equation}\label{eq. 2}
\sum_{\lambda}(\iota_{\lambda})_{\ast} \mbox{Res}_{\varphi_{1}, \varphi_{2}} (\mathcal{F} , \mathcal{N}_{\mathcal{F}}, S_{\lambda} ) = [\varphi_{1}(\mathcal{N}_{12}) \smile  \varphi_{2}(\mathcal{N}_{2})] \frown [M],
\end{equation}
 where $\iota_{\lambda}$ denotes the embedding of $S_{\lambda}$ in $M$.

\end{thm}

\begin{proof}  With the previous notations, the exact sequences
$$ 0 \longrightarrow \mathcal{F}_{2} \longrightarrow \Theta_{M} \longrightarrow
\mathcal{N}_{2} \longrightarrow 0 $$
$$ 0 \longrightarrow \mathcal{F}_{1} \longrightarrow \mathcal{F}_{2} \longrightarrow
\mathcal{N}_{1 2} \longrightarrow 0 $$

\noi induce, respectively, actions $\alpha_{2}$ of $F_{2}^{0}$ on $N_{F_{2}^{0}}$ and $\alpha_{1}$ of $F_{1}^{0}$ on $N_{12}$, [\ref{BB}, \ref{Suwa}].

Let us denote by $\nabla_{12}$ the $F_{1}^{0}$-connection for
$N_{12}$ and $\nabla_{2}$ the $F_{2}^{0}$-connection for
$N_{F_{2}^{0}}.$ Let $S$ be a compact connected component of
$S(\mathcal{F})$ and $U$ a relatively compact open neighborhood of
$S$  in  $M$ disjoint from the other components of $S(\mathcal{F})$.
We set $U_{0} = U \setminus S$ and $U_{1} = U$ and consider the
covering  $\mathcal{U} = \{U_{0}, U_{1} \}$ of $U$. We take
resolutions of the normal sheaves $\mathcal{N}_{12}$ and
$\mathcal{N}_{2}$ by real analytic vector bundles $E_{i}^{12}$ and
$E_{j}^{2}$ on $U$:
$$0 \longrightarrow A_{U}(E_{q}^{1 2}) \longrightarrow \cdots  \longrightarrow A_{U}(E_{0}^{1 2})
\longrightarrow A_{U} \otimes \mathcal{N}_{1 2} \longrightarrow 0 $$
$$0 \longrightarrow A_{U}(E_{r}^{2}) \longrightarrow \cdots  \longrightarrow A_{U}(E_{0}^{2})
\longrightarrow A_{U} \otimes \mathcal{N}_{2} \longrightarrow 0. $$

Since the characteristic class $\varphi_{1}(\mathcal{N}_{12})$ is the characteristic class $\varphi_{1}(\xi^{12})$ of the virtual bundle $\xi^{12} = \sum_{i = 0}^{q} (-1)^{i}E_{i}^{12}$ and $\varphi_{2}(\mathcal{N}_{2}) = \varphi_{2}(\xi^{2}) $ for $\xi^{2} = \sum_{i = 0}^{r} (-1)^{i}E_{i}^{2}$, we define the characteristic class $\varphi(\mathcal{N}_{\mathcal{F}})$, of the  normal sheaf of the flag by $ \varphi_{1}(\mathcal{N}_{12}) \smile \varphi_{2}(\mathcal{N}_{2}).$ Away the from singular component, i.e., $U_{0}$, we have the exact sequences of vector bundles
\begin{equation}\label{eq. 3}
0 \longrightarrow E_{q}^{1 2} \longrightarrow \cdots  \longrightarrow E_{0}^{1 2}
\longrightarrow \mathcal{N}_{1 2} \longrightarrow 0
\end{equation}
\begin{equation}\label{eq. 4}
0 \longrightarrow E_{r}^{2} \longrightarrow \cdots  \longrightarrow E_{0}^{2}
\longrightarrow \mathcal{N}_{2} \longrightarrow 0.
\end{equation}

There exist connections $^{12}\nabla_{0}^{i}$ on $U_{0}$ for each $E_{i}^{12}$   such that the family of connections $$(^{12}\nabla_{0}^{q}, ..., ^{12}\nabla_{0}^{0}, \nabla_{1})$$ is compatible with (\ref{eq. 3}), see [\ref{BB}]. Analogously, there exist connections $^{2}\nabla_{0}^{i}$ on $M$ for each $E_{i}^{2}$ with the same property. We denote $(^{1 2}\nabla_{0}^{q},
\dots , ^{1 2}\nabla_{0}^{0})$ by $^{12}\nabla_{0}^{\bullet}$  and $(^{2}\nabla_{0}^{r}, \dots , ^{2}\nabla_{0}^{0}) $  by  $^{2}\nabla_{0}^{\bullet}$. Then ([\ref{Suwa}, Proposition 8.4]) we have 
\begin{equation} \label{eq.1.4}
\varphi_{1}(^{1 2}\nabla_{0}^{\bullet}) = \varphi_{1}(\nabla^{1})
\ \ \mbox{and} \ \  \varphi_{2}(^{2}\nabla_{0}^{\bullet}) =
\varphi_{2}(\nabla^{2}).
\end{equation}

Let us consider  an arbitrary family $^{1 2}\nabla_{1}^{\bullet} = (^{1 2}\nabla_{1}^{q}, \dots, ^{1
2}\nabla_{1}^{0}) $    of connections  on $U_{1}$, where each  $^{1 2}\nabla_{1}^{(i)}$ is a connection for  $E^{1 2}_{i}$.
Similarly,  we consider an  arbitrary family $^{2}\nabla_{1}^{\bullet} = (^{2}\nabla_{1}^{r}, \dots , ^{2}\nabla_{1}^{0}). $
Then the class $\varphi(\mathcal{N}_{\mathcal{F}}) =
\varphi_{1}(\mathcal{N}_{1 2}) \smallsmile
\varphi_{2}(\mathcal{N}_{2}) = \varphi_{1}(\xi^{1 2}) \smallsmile
\varphi_{2}(\xi^{2})$ in $H^{2(d_{1} + d_{2})}(U; \mathbb{C})$ is represented in  $A^{2(d_{1} +
d_{2})}(U)$  by the cocycle \\

\noi $\varphi(_{2}^{12}\nabla_{\ast}^{\bullet}) = (\varphi_{1}(^{12}\nabla_{0}^{\bullet}),
\varphi_{1}(^{12}\nabla_{1}^{\bullet})  ,
\varphi_{1}(^{12}\nabla_{0}^{\bullet}, ^{12}\nabla_{1}^{\bullet}) )
\smallsmile (\varphi_{2}(^{2}\nabla_{0}^{\bullet}),
\varphi_{2}(^{2}\nabla_{1}^{\bullet})  ,
\varphi_{2}(^{2}\nabla_{0}^{\bullet}, ^{2}\nabla_{1}^{\bullet}) ) 
$ \\

\noi $= (\varphi_{1}(^{12}\nabla_{0}^{\bullet}) \wedge
\varphi_{2}(^{2}\nabla_{0}^{\bullet}),
\varphi_{1}(^{12}\nabla_{1}^{\bullet}) \wedge
\varphi_{2}(^{2}\nabla_{1}^{\bullet})  ,
\varphi_{1}(^{12}\nabla_{0}^{\bullet}) \wedge
\varphi_{2}(^{2}\nabla_{0}^{\bullet}, ^{2}\nabla_{1}^{\bullet}) + \\ \\
+\varphi_{1}(^{12}\nabla_{0}^{\bullet}, ^{12}\nabla_{1}^{\bullet})
\wedge \varphi_{2}(^{2}\nabla_{1}^{\bullet})).$ \\

Then, by the  Theorem \ref{2.15}  we have  $
\varphi(_{2}^{12}\nabla_{\ast}^{\bullet}) \in A^{2(d_{1} +
d_{2})}(U, U_{0})$. Let us consider  $[\varphi(_{2}^{12}\nabla_{\ast}^{\bullet})] =
\varphi_{S}(\mathcal{N}_{\mathcal{F}}, \mathcal{F}) \in H^{2(d_{1} +
d_{2})}(U, U\backslash S ; \mathbb{C})$, we have the class
$$
Res_{\varphi_{1}, \varphi_{2}}
(\mathcal{N}_{\mathcal{F}}, \mathcal{F}; S) = Al(\varphi_{S}(\mathcal{N}_{\mathcal{F}}, \mathcal{F})) \in H_{2n -
2(d_{1} + d_{2})} (S ; \mathbb{C}), $$ where $Al:H^{2(d_{1} +
d_{2})}(U, U\backslash S ; \mathbb{C}) \to  H_{2n -
2(d_{1} + d_{2})} (S ; \mathbb{C})$ is the Alexander   isomorphism (see \cite{Bra}).

\end{proof}

\begin{definition} The class Res$_{\varphi_{1}, \varphi_{2}}
(\mathcal{N}_{\mathcal{F}}, \mathcal{F}; S)$ is called  the Baum-Bott residue for the flag $\mathcal{F}$.

\end{definition}

In the following simple example we  calculate the residues for some symmetric polynomials.
\begin{example}\label{Pn} Consider on $\mathbb{P}^{n}$  the flag $(\F_1,\F_2)$ of foliations, with $\F_1$ induced by the vector field   $\displaystyle X = \partial/\partial z_{3}$ and $\F_2$ induced by the 1-form $\omega = z_{0}dz_{1} - z_{1}dz_{0}$. The singular set of $\F_1$ is given by $S(\mathcal{F}_{1})=\{  [0:0:0:1:0:...:0] \} $ and the singular set of $\F_2$ is $S(\mathcal{F}_{2}) =   S = \{ z_{0} = z_{1} = 0 \} .$
In particular, $ S(\mathcal{F}_{1}) \subset S(\mathcal{F}_{2})$.

 Now,  we calculate the residue for this  flag. Let us denote by $h$ the class $c_1(\mathcal{O}_{\mathbb{P}^{n}}(1))$.
On the one hand  $\mathcal{F}_{1} = \mathcal{O}_{\mathbb{P}^{n}}(1)$ and $\F_2=\mathcal{O}_{\mathbb{P}^{n}}(1)^{\oplus(n-1)}$ we conclude  that
$$c_{1} (\mathcal{N}_{12}) = c_{1}(\mathcal{F}_{2}) - c_{1}(\mathcal{F}_{1}) = (n - 1)h - 1h = (n - 2)h.$$
On the other hand, we have
$c_{1}(\mathcal{N}_{2}) = c_1(T\mathbb{P}^{n})-c_1(\F_2) = 2 h$, since $c_1(T\mathbb{P}^{n})=(n+1)h$.
It follows from the  Theorem \ref{2.16} that, for each $ j = 0,..., n-1$, one has 
$$ \mbox{Res}_{c_{1}^{n - 1 - j}c_{1}^{1 + j} }(\mathcal{F}, \mathcal{N}_{\mathcal{F}}; S ) =  \int_{\mathbb{P}^{n}}  c_{1}^{n - 1 - j}(\mathcal{N}_{12}) c_{1}^{1 + j}(\mathcal{N}_{2}) =(n - 2)^{n - 1 - j}2^{1 + j}.
$$
\end{example}

\begin{rmk}\label{2.17} Note that the Theorem \ref{2.15} is legitime of flags and  "finer"  than  Bott vanishing Theorem, see condition (\ref{eq. 1}). Observe that, with this theorem we can compute the classes:
$$ \varphi (\mathcal{N}_{\mathcal{F}}) = \varphi_{1} (\mathcal{N}_{1,2})\smile \varphi_{1} (\mathcal{N}_{2})$$

\noi with $d_{i} \leq \mbox{codim}( \mathcal{F}_{i})$ for $i = 1,2$ but with $d_{1} + d_{2} > \mbox{codim} (\mathcal{F}_{1})$. An important fact here is that for those polynomials it is not possible to apply the Bott vanishing Theorem. Then in this case, the residue Res$_{\varphi_{1}, \varphi_{2}} (\mathcal{F}, \mathcal{N}_{\mathcal{F}}, S)$ is really   a residue for  flags.

\end{rmk}

Now, we study  a refinement of Theorem \ref{2.16}.

\begin{thm}\label{thm 2.21} With the previous notations, if $d_{1} > \mbox{codim} (\mathcal{F}_{1})$ and $d_{2} > \mbox{codim} (\mathcal{F}_{2})$, then the characteristic class $\varphi(\mathcal{N}_{\mathcal{F}}) = \varphi_{1}(\mathcal{N}_{12})\smile \varphi_{2}(\mathcal{N}_{2})$ is localized at intersection $S := S(\mathcal{F}_{1}) \cap S(\mathcal{F}_{2}) $.

\end{thm}

\begin{proof} We use notations of the proof of Theorem \ref{2.16}. 
Let us denote  by $U_{1}$ a neighborhood of $S(\mathcal{F})$ and  $U_{0} := U_{1} \backslash S := U_{0}^{1} \cup U_{0}^{2}$,  
where $U_{0}^{1} := U_{1} \setminus   S(\mathcal{F}_{2}) $   and $U_{0}^{2}:= U_{1} \setminus   S(\mathcal{F}_{1}) $. 

It is enough to prove the vanishing $\varphi_{1}(^{12}\nabla_{0}^{\bullet}) \wedge
\varphi_{2}(^{2}\nabla_{0}^{\bullet}) = \varphi(_{2}^{12}\nabla_{\ast}^{\bullet})|_{U_{0}} = 0.$
In fact, since   $U_{0} := U_{0}^{1} \cup U_{0}^{2},$  we have
$$\varphi_{2}(^{2}\nabla_{0}^{\bullet})|_{U_{0}^{1}}  =   \varphi_{1}(^{12}\nabla_{0}^{\bullet})|_{U_{0}^{2}} =
\varphi_{2}(^{2}\nabla_{0}^{\bullet}) = \varphi_{1}(^{12}\nabla_{0}^{\bullet}) |_{U_{0}^{1} \cap U_{0}^{2}  }  = 0$$ by Bott vanishing theorem. The remainder of the proof is as Theorem \ref{2.16}.
\end{proof}

Let $M$ be a projective manifold and $H$ an ample divisor on $M$. A torsion free sheaf $\F$ is  semi-stable, with respect to $H$ , if
$$
 \frac{ c_1(\F) \cdot H^{n-1} }{rank(\F)}  \geq  \frac{ c_1(\G) \cdot H^{n-1} }{rank(\G)}
$$
for all subsheaf $\G \subset \F$. In the following corollary we show the relation betwing the semi-stability of the foliation and the positivity of  certain residues. 
\begin{corollary} Let $\mathcal{F}$ be a  holomorphic foliation of dimension two on a   projective manifold with  Picard group  $Pic(X)=\mathbb{Z} \cdot H$, for some  $H$ ample divisor. If $\F$ is semi-stable, then 
$$
Res_{c_{1}^{n}}(\mathcal{F}, (\mathcal{F}/ \mathcal{F}_1) ; S_\lambda)\geq 0
$$
for any subfoliation $\F_1\subset \F$ and any $S_\lambda \subset S( \mathcal{F}_{1}) \cup S(\mathcal{F})$.
\end{corollary}

\begin{proof} 
Let us denote by $c_{1}(\F)=aH$ and $c_{1}(\F_1)=bH$, for a  subfoliation $\F_1$.
Since $\F$ is semi-stable, then 
$$
aH^{n}=c_1(\F) \cdot H^{n-1}\geq  
 \frac{ c_1(\F) \cdot H^{n-1} }{n-1} =   \left(\frac{a}{n-1}\right) H^{n} \geq c_1(\F_1) \cdot H^{n-1}=bH^{n} .
$$
This implies that $0\leq(a-b)$. It follows from   Theorem \ref{2.16}  that 
$$
 0\leq (a-b)^nH^n=c_{1}^{n}(\mathcal{N}_{12})= \sum_{S_\lambda \subset  S(\mathcal{F}_{1}) \cup S(\mathcal{F}_{2})  } Res_{c_{1}^{n}}(\mathcal{F}, (\mathcal{F}/ \mathcal{F}_1); S_\lambda),
$$
Since $c_1(\mathcal{F}/ \mathcal{F}_1)= (a-b)H$. Then, $Res_{c_{1}^{n}}(\mathcal{F}, (\mathcal{F}/ \mathcal{F}_1); S_\lambda)\geq 0$, for any $S_\lambda\subset S(\mathcal{F}_{1}) \cup S(\mathcal{F}). $
\end{proof}

\begin{example}
In the  Example \ref{Pn} we have the flag  given by 
$$\F_1=\mathcal{O}_{\mathbb{P}^{3}}(1)\subset  \F:=\mathcal{O}_{\mathbb{P}^{3}}(1)^{\oplus(2)}\subset T\mathbb{P}^{3}$$ 
and $S=\{z_0=z_1=0\}= Sing(\mathcal{F}, \mathcal{F}_1)$. 
The foliation $\F=\mathcal{O}_{\mathbb{P}^{3}}(1)^{\oplus(2)}$ is semi-stable and 
$$ \mbox{Res}_{c_{1}^{3} }(\mathcal{F}, (\mathcal{F}/ \mathcal{F}_1); S ) =2^{3}>0
$$
for the subfoliation $\F_1.$
\end{example}

\section{Determination and Comparison of certain   residues}

Let $\mathcal{F}$ be a holomorphic foliation on $M$. Although little is known about a formula to express residues in general cases, Baum and Bott in [\ref{BB}] show how to calculate these  residues in specific cases when the degree of the  polynomial $\varphi$ is codim$(\mathcal{F}) + 1$ and additional hypothesis on the singular set. Brunella and Perrone in [\ref{BruPer}] using integration current show a formula for residues for codimension one foliations. Recently Corr\^ea and Fernandez in [\ref{CoFer}] generalize this result  for higher codimensional holomorphic foliations. Here using these  tools and the Godbillon-Vey classes for flags of  Dominguez  [\ref{Domin}], we present a formula for residue of flags. For a basic reference, see [\ref{CoFer}, \ref{Mol}, \ref{Domin}, \ref{BruPer}].

\begin{proposition}\label{prop. 2.3.2} Given a 2-flag $\mathcal{F} = (\mathcal{F}_{1}, \mathcal{F}_{2})$ on a complex manifold $M$. On $M_{0} := M \backslash S(\mathcal{F}_{2}) $ we have Sing$(\mathcal{N}_{1}) \cap M_{0} = \mbox{Sing}(\mathcal{N}_{12}) \cap M_{0}$.
\end{proposition}

\begin{proof} We recall the exact sequence

\begin{equation}\label{eq. 6}
0 \longrightarrow \mathcal{N}_{1,2} \longrightarrow \mathcal{N}_{1} \longrightarrow \mathcal{N}_{2} \longrightarrow 0.
\end{equation}

Taking $p \in M \setminus S(\mathcal{F}_{2})$  the stalk  $\mathcal{N}_{2,p}$ is $\mathcal{O}_{M,p}$ - free. The sequence (\ref{eq. 6}) induces the exact sequence of $\mathcal{O}_{M,p}$-modules
\begin{equation}\label{eq. 7}
0 \longrightarrow \mathcal{N}_{1,2,p} \longrightarrow \mathcal{N}_{1,p} \longrightarrow \mathcal{N}_{2,p} \longrightarrow 0.
\end{equation}

Since $\mathcal{N}_{2,p}$ is a free module then by the  splitting lemma see [\ref{Hat}, p. 147], the sequence (\ref{eq. 7}) splits (here $\mathcal{O}_{M,p}$ is a local ring, then projective and free modules are equivalents):
$$
\mathcal{N}_{1,p} = \mathcal{N}_{1,2,p} \oplus \mathcal{N}_{2,p}. $$
Then,  the module $\mathcal{N}_{1,p}$ is free, if and only if, $\mathcal{N}_{1,2,p}$ is free.

\end{proof}

The proposition implies the following corollary: 

\begin{corollary} If the sheaf $\displaystyle \mathcal{N}_{12} = \mathcal{F}_{2}/\mathcal{F}_{1}$ is locally-free,  then  $S(\mathcal{F}_{1}) \subset S(\mathcal{F}_{2}). $

\end{corollary}

\begin{example}Let $\mathcal{F}$ be the foliation on $\mathbb{P}^{3}$ induced by the polynomial vector field
$$ X = \lambda_{1}z_{1}\frac{\partial }{\partial z_{1}} + \lambda_{2}z_{2}\frac{\partial }{\partial z_{2}} + \lambda_{3}z_{3}\frac{\partial }{\partial z_{3}}, \ \ \ \mbox{with} \ \ \ \lambda_{i} \neq 0 \ \ \ \mbox{for \ \ all} \ \ i. \ \ $$

\noi Consider the osculating planes distribution $\mathcal{F}_{2}$ associated to $X$ (see \cite{Ce}), generated by $X$ and $Y := DX.X$. It is integrable and also given by   the $1$-form logarithmic
$$ \omega = z_1z_2z_3\left(\frac{\lambda_{3} - \lambda_{2}}{\lambda_{1}} \frac{d z_{1}}{z_{1}} + \frac{\lambda_{1} - \lambda_{3}}{\lambda_{2}} \frac{d z_{2}}{z_{2}} + \frac{\lambda_{2} - \lambda_{1}}{\lambda_{3}} \frac{d z_{3}}{z_{3}}\right). $$
In fact, $\mathcal{F} = (\mathcal{F}_{1}, \mathcal{F}_{2})$ is a flag, since   $\omega(X) = 0$. We have the following:
$$ S(\mathcal{F}_{1}) = \Big\{[1:0:0:0], [0:1:0:0], [0:0:1:0], [0:0:0:1] \Big\}. $$

\bc $ S(\mathcal{F}_{2}) =  S = \bigcup S_{ij}$ \ \ for \ \ $i = 0,1,2$, $j = 1,2,3$ \ \ and \ \ $i \neq j$, \ec

\noi where $S_{ij} := \{ z_{i} = z_{j} = 0  \}$.
We observe that $S(\mathcal{F}_{1}) \subset S(\mathcal{F}_{2})$ and that the relative normal sheaf $\displaystyle \mathcal{N}_{12} := \mathcal{F}_{2} / \mathcal{F}_{1}$ is locally-free, since $\mathcal{F}_{1} = \mathcal{O}_{\mathbb{P}^{3}} \subset \mathcal{F}_{2} = \mathcal{O}_{\mathbb{P}^{3}} \oplus \mathcal{O}_{\mathbb{P}^{3}}$ .

\end{example}

\begin{example} Let $ \pi : \mathbb{P}^{3} \dashrightarrow \mathbb{P}^{2}$ be the  rational map given   by $[z_{0}:z_{1}:z_{2}:z_{3}] \longmapsto [z_{0}:z_{1}:z_{2}]$. This is a rational fibration which induces an one-dimensional  foliation on $\mathbb{P}^{3}$ that  we call $\mathcal{F}_{1}$. The singular set of $\mathcal{F}_{1}$ is $S(\mathcal{F}_{1}) = \{ [0:0:0:1] \}$.
Let $\mathcal{G}$ be a codimension one foliation on $\mathbb{P}^{2}$ of degree $d$ with singular set given by $S(\mathcal{G}) = \{ p_{1},...,p_{l} \}$. Now, consider the pull-back of $\mathcal{G}$ by $\pi$ and denote it by $\mathcal{F}_{2} = \pi^{\ast} \mathcal{G}$. We have that $\displaystyle S(\mathcal{F}_{2}) = \bigcup_{p_{i} \in S(\mathcal{G})} \pi^{-1}(p_{i})$.

Since  $\mathcal{F}_{1} = \mathcal{O}_{\mathbb{P}^{3}}(1)$   and $\mathcal{G} = \mathcal{O}_{\mathbb{P}^{2}}(1 - d)$, then $\mathcal{F}_{2} = \mathcal{O}_{\mathbb{P}^{3}}(1 - d) \oplus \mathcal{O}_{\mathbb{P}^{3}}(1)$. The relative sheaf is  $\mathcal{N}_{12}=\mathcal{O}_{\mathbb{P}^{3}}(1 - d)$. In particular, it is  locally free and  $S(\mathcal{F}_{1}) \subset S(\mathcal{F}_{2}) $.
\end{example}

\begin{proposition}\label{prop. 3.5} Let $\mathcal{F} = (\mathcal{F}_{1}, \mathcal{F}_{2} )$ be a flag on a complex manifold $M$ with dim$\mathcal{F}_{1} = \mbox{codim}\mathcal{F}_{2} = 1$. Then $\mathcal{F}_{1}$ has no isolated singularities in $M \diagdown S (\mathcal{F}_{2})$.

\end{proposition}

\begin{proof} The situation is local. Suppose that $p$ is an isolated singularity of $\mathcal{F}_{1}$ and take  a neighborhood $\{ U, (z_{1},...,z_{n}) \}$ of $p$, where $\mathcal{F}_{2}|_{U}$ is regular. On this open subset we can consider $\mathcal{F}_{2}$ as induced by an 1-form $\omega = dz_{1}$ and $\mathcal{F}_{1}$ by a vector field $X = \sum_{i = 1}^{n} f_{i} dz_{i}$. 
\noi Since   $\mathcal{F} = (\mathcal{F}_{1}, \mathcal{F}_{2} )$ is a flag, we have
$ \iota_{X} \omega = f_{1}=0.$
But this show that $$S(\mathcal{F}_{1})|_{U} = \{ f_{2} = \cdots  = f_{n} = 0
 \}$$ which cannot be  isolated.
\end{proof}

\begin{corollary} \label{sing_iso} Let $\mathcal{F} = (\mathcal{F}_{1}, \mathcal{F}_{2} )$ be a flag on a compact complex manifold $M$ with dim$\mathcal{F}_{1} = \mbox{codim}\mathcal{F}_{2} = 1$. If $S_{0}(\mathcal{F}_{i})$ denotes the isolated singularities of the foliation $\mathcal{F}_{i}$ for $i = 1,2$, we have that $S_{0}(\mathcal{F}_{1}) = S_{0}(\mathcal{F}_{2})$.

\end{corollary}

\begin{proof} See Proposition \ref{prop. 3.5} and [\ref{Mol}, Corollary 1, pg 778]. \\
\end{proof}

\begin{proposition} For a flag $\mathcal{F} = (\mathcal{F}_{1}, \mathcal{F}_{2} )$ on $M$ with dim$(\mathcal{F}_{1}) = \mbox{codim}(\mathcal{F}_{2}) = 1$ and $S(\F)$ admitting isolated singularities we have
$$ \mbox{Res}_{c_{n}} (\mathcal{F}_{2}, \mathcal{N}_{2},p) = (-1)^{n} (n-1)! \mbox{Res}_{c_{n}} (\mathcal{F}_{1}, \mathcal{N}_{1},p). $$

\end{proposition}

\begin{proof} It follows from Corollary \ref{sing_iso} that  $S(\F)=S_{0}(\mathcal{F}_{1}) = S_{0}(\mathcal{F}_{2})$.  Let $p \in S(\mathcal{F}_{1}) \cap S(\mathcal{F}_{2})$ be an isolated singulary. We know that near the point  $p$ we can consider $\mathcal{F}_{1}$ as induced by a vector field $X = \sum f_{i} \partial / \partial z_{i}$ and $\mathcal{F}_{2}$ by an 1-form $\eta = \sum g_{i}d z_{i}$. On the one hand,  Res$_{c_{n}} (\mathcal{F}_{1}, \mathcal{N}_{1}; p ) = \mu (f; p)$ is the Milnor number of $ f = (f_{1},\dots ,f_{n})$ at $p$, see \cite{Suwa}. On the other hand, we have
$$Res_{c_{n}} (\mathcal{F}_{2}, \mathcal{N}_{2}; p ) = (-1)^{n}(n-1)! \mu(g; p),$$
where $g = (g_{1}, \dots ,g_{n})$ and $n = \dim_{\mathbb{C}} M$, see Suwa [\ref{Suwa1}, Proposition 3.12]. Since $\mathcal{F} = (\mathcal{F}_{1}, \mathcal{F}_{2})$ is a flag we have
\begin{equation} \label{eq. 8}
\iota_{X} \eta = \sum f_{i} g_{i} = 0.
\end{equation}
We claim that $(f_{1},\dots ,f_{n}) = (g_{1}, \dots ,g_{n})$ as generated ideals:  \\

\noi Let us  consider the exact Koszul complex of the regular sequence $ (f_{1},\dots ,f_{n})$,
$$\displaystyle 0  \longrightarrow \bigwedge^{n} \mathcal{O}^{n} \longrightarrow \cdots \longrightarrow \bigwedge^{2} \mathcal{O}^{n} \stackrel{r}{\longrightarrow} \mathcal{O}^{n} \stackrel{s}{\longrightarrow}
\mathcal{O} \longrightarrow 0, $$
\noi where $r(e_{i}\wedge e_{j}) = f_{i}e_{j} - f_{j}e_{i}$ and $s(e_{i}) = f_{i}$. From (\ref{eq. 8}) one has that $ (g_{1}, \dots ,g_{n}) \in \mbox{Ker} (s) = \mbox{Im} (r) $, then

$$ r(\sum P_{ij} e_{i} \wedge e_{j}) = \sum P_{ij} (f_{i}e_{j} - f_{j}e_{i}) = \sum g_{i}e_{i}. $$

This implies that $(g_{1},\dots ,g_{n}) \subset (f_{1},\dots ,f_{n})$. If we  consider the Koszul complex of $(g_{1},\dots ,g_{n})$ we obtain  the equality of ideals.
Therefore $\mu(f; p) = \mu(g; p)$ and $$\mbox{Res}_{c_{n}}  (\mathcal{F}_{2}, \mathcal{N}_{2}; p) = (-1)^{n}(n - 1)! \mbox{Res}_{c_{n}} (\mathcal{F}_{1}, \mathcal{N}_{1}; p). $$
\end{proof}

The following  example is   inspired from  an example due to  Izawa in [\ref{Iza}, Example 5, pg 907].

\begin{example}\label{ex.2.3.8} Let $Y := \mathbb{P}^{5} \times \mathbb{P}^{1}$ be  with homogeneous coordinates 
\noi $([x_{0}:\cdots :x_{5}],[y_{0}:y_{1}])$. We consider the  regular foliation on $Y$ given by $\widetilde{\mathcal{G}} := \pi^{-1} \Omega_{\mathbb{P}^{1}}$, where $\pi$ is the standard projection of $\mathbb{P}^{5} \times \mathbb{P}^{1}$ in $\mathbb{P}^{1}$. Let
$$ X := V(x_{0}^{l} + x_{1}^{l} + x_{2}^{l} + x_{3}^{l} + x_{4}^{l} + x_{5}^{l})\cap V(x_{0}y_{0} + x_{1}y_{1}), \ \ \ l \in \mathbb{Z}_{+}$$

\noi be a  regular sub-manifold of $Y$. We consider the inclusion map $i : X \longrightarrow Y$. Let us denote by  $\mathcal{F}_{2} = i^{-1} \widetilde{\mathcal{G}}$, the inverse image of $\widetilde{\mathcal{G}}$, which defines a singular foliation of codimension one on $X$. In this case, the non-transversal loci of $i$ to $\widetilde{\mathcal{G}}$ determines $S(\mathcal{F}_{2})$, the singular set of foliation $\mathcal{F}_{2}$. In order to determinate  the non-transversal points, we take the inhomogeneous coordinates over $x_{0} \neq 0$ and $y_{0} \neq 0$ as $\displaystyle (s,x,y,w,t) = (\frac{x_{1}}{x_{0}}, \frac{x_{2}}{x_{0}}, \frac{x_{3}}{x_{0}}, \frac{x_{4}}{x_{0}}, \frac{x_{5}}{x_{0}} )$ and $\displaystyle z = (\frac{y_{1}}{y_{0}})$. With these coordinates we can express, locally, the vector field $X$ by
$$ X = \{ (s,x,y,w,t;z) ; 1 + x^{l} + y^{l} + w^{l} + t^{l} = 0 \ \ \mbox{and} \ \ 1 + sz = 0   \}.$$
Thus, we have that $\displaystyle z = - (-1)^{\frac{-1}{l}}(1 + x^{l} + y^{l} + w^{l} + t^{l})^{\frac{-1}{l}}$. We know that $\mathcal{F}_{2}$ is given by the  $1$-form
$$ \omega = dz = \frac{\partial z}{\partial x} dx + \frac{\partial z}{\partial y} dy + \frac{\partial z}{\partial w} dw + \frac{\partial z}{\partial t} dt.$$
\noi Here, we use the following notation for coordinates of the  $1$-form that induces $\mathcal{F}_{2}$ \\
\bc
\noi $\displaystyle\zeta_{1} = \frac{\partial z}{\partial x} = (-1)^{\frac{-1}{l}} x^{l - 1} (1 + x^{l} + y^{l} + w^{l} +t^{l})^{\frac{-l - 1}{2}}$
\ec

\bc
\noi $\displaystyle\zeta_{2} = \frac{\partial z}{\partial y} = (-1)^{\frac{-1}{l}} y^{l - 1} (1 + x^{l} + y^{l} + w^{l} +t^{l})^{\frac{-l - 1}{2}} $
\ec

\bc
\noi $\displaystyle\zeta_{3} = \frac{\partial z}{\partial w} = (-1)^{\frac{-1}{l}} w^{l - 1} (1 + x^{l} + y^{l} + w^{l} +t^{l})^{\frac{-l - 1}{2}} $
\ec

\bc
\noi $\displaystyle\zeta_{4} = \frac{\partial z}{\partial t} = (-1)^{\frac{-1}{l}} t^{l - 1} (1 + x^{l} + y^{l} + w^{l} +t^{l})^{\frac{-l - 1}{2}}.$
\ec

Since the $z$-axis is the transversal direction for the leaves of $\mathcal{F}_{2}$, the non-transversal conditions are
given by $\zeta_{1} = \zeta_{2} = \zeta_{3} = \zeta_{4} = 0$ such that $(x,y,w,t) = (0, 0)$. Thus, with the defining equations, we see that the non-transversal points are given by
$$(s, x, y, w, t; z) = (\omega_{k}, 0, 0, 0, 0;-\omega_{l - k - 1})_{k = 0,..., l -1},$$
\noi where we denote by $\omega_{k}$ the $l$-roots of $-1$. Therefore, the singular set of $\mathcal{F}_{2}$ is given by these points. Consider the one dimensional foliation on $X$, denoted by $\mathcal{F}_{1}$ and  given locally by following vector the field $ X = (X_{1}, X_{2}, X_{3}, X_{4})$, where \\

\bc
\noi $X_{1} = (-1)^{\frac{-1}{l}} (-y^{l - 1}) (1 + x^{l} + y^{l} + w^{l} +t^{l})^{\frac{-l - 1}{2}} = -\zeta_{2}$ \\\ec

\bc
\noi $X_{2} = (-1)^{\frac{-1}{l}} x^{l - 1} (1 + x^{l} + y^{l} + w^{l} +t^{l})^{\frac{-l - 1}{2}} = \zeta_{1}$ \\
\ec

\bc
\noi $X_{3} = (-1)^{\frac{-1}{l}} (-t^{l - 1}) (1 + x^{l} + y^{l} + w^{l} +t^{l})^{\frac{-l - 1}{2}} = - \zeta_{4}$ \\
\ec

\bc
\noi$X_{4} = (-1)^{\frac{-1}{l}} w^{l - 1} (1 + x^{l} + y^{l} + w^{l} +t^{l})^{\frac{-l - 1}{2}} = \zeta_{3}.$
\ec

\noi Note that $\mathcal{F} = (\mathcal{F}_{1}, \mathcal{F}_{2})$ is in fact a flag, because the following holds
$$ X_{1}\zeta_{1} + X_{2}\zeta_{2} + X_{3}\zeta_{3} + X_{4}\zeta_{4} = 0.$$

We observe that $S(\mathcal{F}_{1}) = S(\mathcal{F}_{2})$. 

Now, 
consider the case of the class $c_4$. One has, on the one hand, by using  the local coordinates of the vector field and $1$-form above: \\

\noi Res$\displaystyle _{c_{4}} (\mathcal{F}_{2}, \mathcal{N}_{2}; p) = (-1)^{4} 3 ! \left(\frac{1}{2 \pi i}\right)^{4} \int_{T} \frac{d\zeta_{1} \wedge d\zeta_{2} \wedge d\zeta_{3} \wedge d\zeta_{4}  }{ \zeta_{1}\zeta_{2}\zeta_{3}\zeta_{4}}= $ \\\

$\hspace{0.2cm}\displaystyle = \int_{T}\Big((l - 1)^{2} + (l^{2} - 1)\frac{x^{l} + y^{l} + w^{l} + t^{l} }{1 + x^{l} + y^{l} + w^{l} + t^{l}}\Big)\frac{dx}{x} \wedge \frac{dy}{y} \wedge \frac{dw}{w} \wedge \frac{dt}{t} = 6. (l - 1)^{2}$, \\

\noi where $T$ is given by $ \{ |x| = |y| = |w| = |t| = \epsilon \} $. 

On the other hand, since
$$ \frac{dX_{1} \wedge dX_{2} \wedge dX_{3} \wedge dX_{4}   }{X_{1}.X_{2}.X_{3}.X_{4}} = \frac{d(-\zeta_{2})\wedge d(\zeta_{1}) \wedge d(-\zeta_{4}) \wedge d(\zeta_{3})    }{ (-\zeta_{2})\zeta_{1}(-\zeta_{4})\zeta_{3}  } = \frac{d\zeta_{1} \wedge d\zeta_{2} \wedge d\zeta_{3} \wedge d\zeta_{4}  }{ \zeta_{1}\zeta_{2}\zeta_{3}\zeta_{4}}.$$
we get
$$\mbox{Res}\displaystyle _{c_{4}} (\mathcal{F}_{1}, \mathcal{N}_{1}; p) = \left(\frac{1}{2 \pi i}\right)^{4} \int_{T} \frac{dX_{1} \wedge dX_{2} \wedge dX_{3} \wedge dX_{4}   }{X_{1}X_{2}X_{3}X_{4}} = (l - 1)^{2}.$$

\noi Therefore,

$$Res_{c_{4}} (\mathcal{F}_{2}, \mathcal{N}_{2}; p) = 3!(l - 1)^{2} = 3! \mbox{Res}_{c_{4}}(\mathcal{F}_{1}, \mathcal{N}_{1}; p). $$

\end{example}

\subsection{Determination of certain Baum-Bott residues}\label{subsect3.1}

Let $\mathcal{F} = (\mathcal{F}_{1}, \mathcal{F}_{2})$ be a $2$-flag on a compact complex manifold $M$ of dimension $n$. We denote by $(k_{1}, k_{2})$ the codimensions of this flag and  Sing$_{k_{i}+1}(\mathcal{F}_{i})$ the set of irreducible components of $S(\mathcal{F}_{i})$ of pure codimension $k_{i} + 1$. Let
$$ S_{\ast}(\mathcal{F}) := \mbox{Sing}_{k_{1} + 1}(\mathcal{F}_{1}) \cup \mbox{Sing}_{k_{2} + 1}(\mathcal{F}_{2}), \ \ \ \ \ M^{0}:= M \setminus S(\mathcal{F}) \ \ \ \mbox{e} \ \ \ M^{\ast} := M \setminus S_{\ast} (\mathcal{F}).$$

In this section, we will  show that the characteristic classes  $c_{1}^{k_{1}-j + 1}(\mathcal{N}_{12}) c_{1}^{j}(\mathcal{N}_{2}) $ can be localized on Sing$_{k_{1}+1} (\mathcal{F}_{1}).$ As a consequence we obtain  a relation  between the flags residues and involving foliations residue.

On $M^{0}=M-Sing (\mathcal{F}_{1})$   there exist local  forms $\omega^{2}_{\alpha}$ and $\omega^{12}_{\alpha}$ where $\omega^{2}_{\alpha}$ is an $k_{2}-$ form that induces $\mathcal{F}_{2}$ and $\omega^{12}_{\alpha}$ is an $(k_{1} - k_{2})-$ form such that $\omega^{1}_{\alpha} := \omega^{2}_{\alpha} \wedge \omega^{12}_{\alpha}$ induces  $\mathcal{F}_{1}$ satisfying : \\

\noi 1) Decomposability:
$$\omega^{2}_{\alpha} = \eta_{1}^{\alpha} \wedge ... \wedge \eta_{k_{2}}^{\alpha} \ \ \mbox{and} \ \ \omega^{12}_{\alpha} = \eta_{k_{2} + 1}^{\alpha} \wedge ... \wedge \eta_{k_{1}}^{\alpha}$$

\noi 2) Integrability condition: There are matrices of 1-forms $(\theta_{u v}^{\alpha}), (\theta_{a v}^{\alpha}) \ \ \mbox{and} \ \ (\theta_{a b}^{\alpha})$ with $1 \leq u, v \leq k_{2} \ \ \mbox{and} \ \ k_{2} + 1 \leq a, b \leq k_{1}$ such that
$$ d \eta_{u}^{\alpha} = \sum_{v=1}^{k_{2}} \theta_{u v}^{\alpha}\wedge \eta_{v}^{\alpha} \ \ \mbox{and} \ \ d \eta_{a}^{\alpha} = \sum_{v=1}^{k_{2}} \theta_{a v}^{\alpha}\wedge \eta_{v}^{\alpha} + \sum_{b= k_{2} + 1}^{k_{1}} \theta_{a b}^{\alpha}\wedge \eta_{b}^{\alpha}. $$
We define $\theta^{2}_{\alpha} = \sum_{u=1}^{k_{2}} (-1)^{u + 1} \theta_{u u}^{\alpha}, \ \  \theta^{12}_{\alpha} = \sum_{a=k_{2}+ 1}^{k_{1}} (-1)^{a + 1} \theta_{a a}^{\alpha}$ and put $\theta^{1}_{\alpha} := \theta^{2}_{\alpha} + \theta^{12}_{\alpha}.$

Take   an irreducible component $Z \in \mbox{Sing}_{k_{1} + 1}(\mathcal{F}_{1})$ and a generic point $p \in Z$. Pick $B_{p}$ a small ball centered at $p$ such that $S(B_{p}) \subset B_{p}$ is a sub-ball of dimension $n - k_{1} - 1$ (same dimension than the component $Z$). As above, we define $B_p^{\ast} := B_p \setminus S_{\ast} (\mathcal{F})$

 Let us consider $\omega_{2} = \eta_{1} \wedge  \cdots \wedge \eta_{k_{2}} \ \ \mbox{and}  \ \ \omega_{12} = \eta_{k_{2} + 1} \wedge ... \wedge \eta_{k_{1}}$, with $\omega_{1} = \omega_{2} \wedge \omega_{12}$, local generators as above. Take smooth sections of $\mathcal{N}_{12}^{\ast}$ and $\mathcal{N}_{2}^{\ast}$ instead of holomorphic ones, then the cohomology groups $H^{1}(B_{p}^{\ast}, \mathcal{N}_{12}^{\ast} )$ and  $H^{1}(B_{p}^{\ast}, \mathcal{N}_{2}^{\ast} )$ are trivial. Then it is possible to find matrices of (1,0)-forms $(\theta_{u v}), (\theta_{a v}) \ \ \mbox{and} \ \ (\theta_{a b})$ such that
$$ d \eta_{u} = \sum \theta_{u v}\wedge \eta_{v} \ \ \mbox{and} \ \ d \eta_{a} = \sum \theta_{a v}\wedge \eta_{v} + \sum \theta_{a b}\wedge \eta_{b}. $$
We define $\theta^{2} = \sum (-1)^{u + 1} \theta_{u u} \ \ \mbox{and} \ \ \theta^{12} = \sum (-1)^{a + 1} \theta_{a a }$.
Dominguez showed in  [\ref{Domin}, Theorem  5.2] that the forms
$$\psi_{j} := (2 \pi i)^{-k_{1} - 1}\theta^{12}\wedge (d \theta^{2})^{j} \wedge (d \theta^{12})^{k_{1} - j}$$
are closed in de Rham cohomology. These forms correspond to cohomology classes in $H^{\ast} (B_{p}^{\ast}, \mathbb{C})$.
Then the de Rham class can be integrated over an oriented $(2k_{1} + 1)$-sphere $L_{p} \subset B_{p}^{\ast}$ and it defines the Baum-Bott residue of $\mathcal{F}$ at $Z$ (see \cite{BSS,Suwa}) 
$$ BB^{j}(\mathcal{F}, Z) := (2 \pi i)^{-k_{1} - 1} \int_{L_{p}}\psi_{j},\ \ \ \mbox{for each} \ \ \ 0 \leq j \leq k_{2}. $$

\begin{thm}\label{c_1} Let $\mathcal{F} = (\mathcal{F}_{1}, \mathcal{F}_{2})$ be a 2-flag of codimensions $(k_{1}, k_{2})$ on a compact complex manifold $M$. If codim S$(\mathcal{F}) \geq k_{1} + 1$, then for each $0 \leq j \leq k_{2}$, we have

$$c_{1}^{k_{1} - j + 1}(\mathcal{N}_{12}) \smile c_{1}^{j}(\mathcal{N}_{2}) = \sum_{Z \subset  \mbox{Sing}_{k_{1}+1}(\mathcal{F}_{1}) \cup \mbox{Sing}_{k_{1}+1}(\mathcal{F}_{2}) } \lambda_{Z}(\mathcal{F})[Z], $$

\noi where $\lambda_{Z}(\mathcal{F}) = BB^{j}(\mathcal{F}, Z)$.
\end{thm}

\begin{proof} The flag $\mathcal{F} = (\mathcal{F}_{1}, \mathcal{F}_{2})$ can be locally defined on the open  $U_\alpha$ by $\omega_{2} = \eta_{1} \wedge ... \wedge \eta_{k_{2}}, \ \ \, \omega_{12} = \eta_{k_{2} + 1} \wedge ... \wedge \eta_{k_{1}} \ \ \mbox{and} \ \ \omega_{1} = \omega_{2} \wedge \omega_{12}$ as above. Then, we can find matrices of (1,0)-forms $(\theta_{u v}^{\alpha}), (\theta_{a v}^{\alpha}) \ \ \mbox{and} \ \ (\theta_{a b}^{\alpha})$ with $\theta_{i j}^{\alpha} \in A^{1,0} (B_{p}^{\ast})$ such that

$$d \eta_{u} = \sum \theta_{u v}^{\alpha} \wedge \eta_{v} \ \ \mbox{and} \ \ d \eta_{a} = \sum \theta_{a v}^{\alpha} \wedge \eta_{v} + \sum \theta_{a b}^{\alpha} \wedge \eta_{b}$$

We say that $\nabla = \begin{pmatrix}\theta_{u v}^{\alpha}  & 0                \\
                    \theta_{a v}^{\alpha}  & \theta_{a b}^{\alpha}    \end{pmatrix}$
represents the curvature matrix of flag $\mathcal{F}$.  Let us consider a neighborhood  $V$ of $S_{\ast}(\mathcal{F})$, then we can define $\theta_{\alpha}^{2} = \sum (-1)^{u + 1} \theta_{u u}^{\alpha}$, and  $\theta_{\alpha}^{12} = \sum (-1)^{a + 1} \theta_{a a}^{\alpha}$.

Let us consider $\Theta^{2} := (2 \pi i)^{-1} d  \theta_{\alpha}^{2}$ and $\Theta^{12} := (2 \pi i)^{-1} d\theta_{\alpha}^{12}$ closed forms globally defined which represent in de Rham cohomology the Chern classes of $ \mathcal{N}_{2}$ and $\mathcal{N}_{12}$, respectively. Therefore $(\Theta^{2})^{j} \wedge (\Theta^{12})^{k_{1} - j + 1}$ represent $c_{1}^{k_{1} - j + 1}(\mathcal{N}_{12})\smile c_{1}^{j}(\mathcal{N}_{2})$, and moreover, by  the
vanishing Theorem \ref{2.15}    we have
 $$ \mbox{Supp}(c_{1}^{k_{1} - j + 1}(\mathcal{N}_{12})\smile c_{1}^{j}(\mathcal{N}_{2})) \subset \overline{V}. $$

Take $T_{1} \subset M$ a ball of real dimension $2(k_{1} + 1)$ intersecting transversally Sing$_{k_{1} + 1}(\mathcal{F}_{1})$ at a single point $p \in Z$, with $V \cap T \Subset T$. Then, by the  Stokes formula \\

\noi $BB^{j}(\mathcal{F}, Z) = (2 \pi i)^{k_{1} + 1} \int_{\partial T_{1}} \theta^{12}_{\alpha} \wedge (d \theta^{2}_{\alpha})^{j} \wedge (d \theta^{12}_{\alpha})^{k_{1} - j}
=$
\begin{equation}\label{eq.4.1}
= (2 \pi i)^{k_{1} + 1} \int_{ T_{1}} (d \theta^{2}_{\alpha})^{j} \wedge (d  \theta^{12}_{\alpha})^{k_{1} - j + 1}.
\end{equation}

This means that the $2(k_{1} + 1)$-form $(\Theta^{2})^{j} \wedge (\Theta^{12})^{k_{1} - j + 1} = (d \theta^{2}_{\alpha})^{j} \wedge (d  \theta^{12}_{\alpha})^{k_{1} - j + 1}$ is cohomologous, as a current, to the integration current over $BB^{j}(\mathcal{F},Z)[Z],$ i.e.,

$$  c_{1}^{k_{1} - j + 1}(\mathcal{N}_{12})\smile c_{1}^{j}(\mathcal{N}_{2}) = \sum_{Z} BB^{j}(\mathcal{F}, Z) [Z].$$

\end{proof}

This theorem answers, partially, to the question: \textit {How to calculate the residue for flags?}

In fact, the Theorem \ref{c_1} says  us that

$$ \mbox{Res}_{c_{1}^{k_{1} - j + 1}c_{1}^{j}} (\mathcal{F}, \mathcal{N}_{\mathcal{F}}; Z) = \alpha_{\ast}(   BB^{j}(\mathcal{F}; Z)[Z] ), $$
 where $\alpha_{\ast}$ is the Poincar\'e duality isomorphism (see \cite{Bra})
$$  H^{2(k_{1} + 1 )}(M; \mathbb{C})
\stackrel{\alpha_{\ast}}{\longrightarrow}
H_{2(n - k_{1} - 1)}(M; \mathbb{C}) .
$$
\begin{corollary} If either $\det(\mathcal{N}_{12})$ or $\det(\mathcal{N}_{2})$ is ample then there exist at least  an irreducible component $Z \in S(\mathcal{F}_{1})$ of codimension $k_{1}+1$.
\end{corollary}

In the following, we prove a formula that compare the (sum)residue of flag with the residue of foliation.

\begin{corollary}\label{3.10} For each $Z \in \mbox{Sing}_{k_{1} + 1}(\mathcal{F}_{1})$ and hypotheses above we have
\begin{equation}\label{eq.2.2}
\sum_{j = 0}^{k_{2}} \binom{k_{1} + 1}{j} BB^{j}(\mathcal{F}, Z) = BB(\mathcal{F}_{1},Z),
\end{equation}

\noi where the term of the right hand side of (\ref{eq.2.2}) is the  "Baum-Bott residue" for $\mathcal{F}_{1}$.
\end{corollary}

\begin{proof}  By Dominguez [\ref{Domin}, Remarque 1], we have
$$ \sum_{j = 0}^{k_{2}} \binom{k_{1} + 1}{j} [\theta^{12}\wedge (d \theta^{2})^{j} \wedge (d \theta^{12})^{n - 1 - j}] = [\theta^{1} \wedge (d \theta^{1})^{n - 1}]$$
\noi in   de Rham cohomology, where we note  $\theta^{1} = \theta^{2} + \theta^{12}.$ \\

Thus one has 
$$ \sum_{j = 0}^{k_{2}} \binom{k_{1} + 1}{j} \theta^{12}\wedge (d \theta^{2})^{j} \wedge (d \theta^{12})^{n - 1 - j} - \theta^{1} \wedge (d \theta^{1})^{n-1} = d\sigma$$ for some  differential form  $\sigma$. Now, integrated over a sphere $\partial T_{1}$ as above, we have
$$ \sum_{j = 0}^{k_{2}} \binom{k_{1} + 1}{j} BB^{j}(\mathcal{F},Z) = BB(\mathcal{F}_{1},Z),$$
proving  the corollary.
\end{proof}

\begin{corollary}\label{coro.2.4.5} Let $\mathcal{F} = (\mathcal{F}_{1}, \mathcal{F}_{2})$ be a flag with $\dim (\mathcal{F})_{1} = \mbox{codim} (\mathcal{F}_{2}) = 1$ and assume that the singular set of the  flag is composed with  isolated singularities only. Then, we have
$$ \mbox{Res}_{c_{1}^{n}}(\mathcal{F}, \mathcal{N}_{12}; p) = \mbox{Res}_{c_{1}^{n}}(\mathcal{F}_{1}, \mathcal{N}_{1}; p),$$
 where $p \in S(\mathcal{F}) = S(\mathcal{F}_{1}) = S(\mathcal{F}_{2})$.
\end{corollary}

\begin{proof} By Corollary \ref{3.10} and hypothesis   $k_{1} = n - 1$ and $k_{2} = 1$ we have
$$ BB^{0}(\mathcal{F}; p) + nBB^{1}(\mathcal{F}; p) = BB(\mathcal{F}_{1}; p). $$
This   and the hypothesis of isolated  singularities   imply that
$$ \mbox{Res}_{c_{1}^{n}}(\mathcal{F}, \mathcal{N}_{12}; p) + \mbox{Res}_{c_{1}^{n -1 }, c_{1}}(\mathcal{F}, \mathcal{N}_{\mathcal{F}}; p) = \mbox{Res}_{c_{1}^{n}}(\mathcal{F}_{1}, \mathcal{N}_{1}; p),$$
where $\displaystyle \mbox{Res}_{c_{1}^{n }, c_{1}}(\mathcal{F}, \mathcal{N}_{\mathcal{F}}; p) = \left(\frac{1}{2 \pi i}\right)^{ n - 1} \int_{L_{p}} \theta^{12} \wedge (d \theta^{2})^{j} \wedge (d \theta^{12})^{n - 1 - j}$

\noi with $\theta^{2}$ an $(1,0)$-form such that if $\omega$ is the $1$-form that induces locally $\mathcal{F}_{2}$, we have
$$ d \omega = \theta^{2} \wedge \omega. $$
By Malgrange, see [\ref{Mal}, Th\'eor\`eme 0.I, pg 163] we have $\omega$ admits an integral factor, i.e., there are holomorphic functions $f$ and $g$ with $f(p) \neq 0$ such that $\omega = fdg$, this implies that
$$ d\omega = df \wedge dg = \frac{df}{f} \wedge (f. dg) = \frac{df}{f} \wedge \omega.$$
Then, we can consider $\displaystyle \theta^{2} = \frac{df}{f} = d (\log f)$. Since this is an exact form,   $d \theta^{2} = 0$ and  \hfill\break 
$ \mbox{Res}_{c_{1}^{n -1 }, c_{1}}(\mathcal{F}, \mathcal{N}_{\mathcal{F}}; p) = 0$. Therefore, the result is proved.

\end{proof}

\begin{example} Let $\mathcal{F} = (\mathcal{F}_{1}, \mathcal{F}_{2})$ be the flag on the manifold $X \subset \mathbb{P}^{5} \times \mathbb{P}^{1}$ of   Example \ref{ex.2.3.8}. By the corollary \ref{coro.2.4.5} we have: 
$$ \mbox{Res}_{c_{1}^{4}}(\mathcal{F}, \mathcal{N}_{12}; p) = \mbox{Res}_{c_{1}^{4}}(\mathcal{F}_{1}, \mathcal{N}_{1}; p),$$
 where $$\displaystyle \mbox{Res}_{c_{1}^{4}}(\mathcal{F}_{1}, \mathcal{N}_{1}; p) = (\frac{1}{2 \pi i})^{4} \int_{T} tr(JX)^{4} \frac{dx\wedge dy \wedge dw \wedge dt}{X_{1}X_{2}X_{3}X_{4}}$$ and $tr(JX)$ denotes the trace of the jacobian of $X$.
 We can check that $tr (JX) = 0$. Therefore, we conclude that
$$ Res_{c_{1}^{4}}(\mathcal{F}, \mathcal{N}_{12};p) = 0. $$

\end{example}

\section{Nash residue for flags and rationality of the residues}

\begin{RCF}\label{2.20} Let $\mathcal{F} = (\mathcal{F}_{1}, \mathcal{F}_{2} )$ be a 2-flag of holomorphic foliations on a complex manifold $M$.
 Let $S$ be a compact connected component of the singular set of the  flag and $\varphi = (\varphi_{1}, \varphi_{2} )$, where $\varphi_{1}$ and  $\varphi_{2}$  are   homogeneous symmetric polynomials, of degrees $d_{1}$ and
$d_{2}$, such that at least one of the inequalities 
$$
d_{1} > n - \dim(\F_{1})+1 \ \ \ or \ \ \ d_{2} > n - \dim(\F_{2})+1
\ \ \ or \ \ \ d_{1} + d_{2}
> n - \dim(\F_{1})+1
$$
is satisfied. 
If $\varphi_{1}$ and  $\varphi_{1}$  have  rational coefficients, then
$$ \mbox{Res}_{\varphi_{1}, \varphi_{2}} (\mathcal{F} , \mathcal{N}_{\mathcal{F}}, S ) \in H_{2n - 2(d_{1} + d_{2} )} (S; \mathbb{Q}).$$

\end{RCF}

In this section, we give  a partial answer for this conjecture. We  will consider
 the Nash modification of flags of holomorphic foliations. For the  ordinary case see [\ref{BB},\ref{BS}]. We analyze  the relation  between the  Nash residue
 and the  Baum-Bott residue.

\subsection{ Nash Contruction}

Let $M$ be a complex manifold of dimension $n$ and $\mathcal{F} =
(\mathcal{F}_{1}, \mathcal{F}_{2})$ an 2-flag of singular
holomorphic foliations of dimensions $(q_{1}, q_{2})$ on $M$. For  each point $x \in M$, we set
$$ F_{i}(x) = \{ v(x) / v \in \mathcal{F}_{i,x}     \} \subset
T_{x}M$$
which  is an $q_{i}$-dimensional subspace if and only if $x \notin
S(\mathcal{F})$, for $i = 1,2.$ Thus we obtain  a flag of subspaces
$F_{1}(x) \subset F_{2}(x) \subset T_{x}M$ for each point $x \in M
\setminus S(\mathcal{F}).$

We will define the flag bundle using the
Grassmann bundles of $q_{i}$-planes. Let $\widetilde{\pi}_{2} :
G_{q_{2}}(TM) \longrightarrow M $ the  Grassmann bundle of $q_{2}$-planes
in $TM$. We have the Nash modification $M_{2}^{\nu} = \overline{Im  \sigma_{2}}$ of $M$ with respect to
$\mathcal{F}_{2}$,
where $\sigma_{2}$ is the  natural section of $\widetilde{\pi}_{2}$  induced by
$\mathcal{F}_{2}$. We have the exact sequence on $M_{2}^{\nu}$:
\begin{equation}\label{eq.3.1}
0 \longrightarrow T_{2}^{\nu} \longrightarrow \pi_{2}^{\ast} TM
\longrightarrow N_{2}^{\nu} \longrightarrow 0.
\end{equation}
Analogously we consider the Grassmann bundle of $q_{1}$-planes 
in $TM$ denoted by $ \widetilde{\pi}_{1} : G_{q_{1}}(TM)
\longrightarrow M $ and we obtain the Nash modification of $M$ with
respect to $\mathcal{F}_{1}, \ \ M_{1}^{\nu} = \overline{Im
\sigma_{1}},$ and the exact sequence on $M_{1}^{\nu}:$
\begin{equation}\label{eq.3.2}
0 \longrightarrow T_{1}^{\nu} \longrightarrow \pi_{1}^{\ast} TM
\longrightarrow N_{1}^{\nu} \longrightarrow 0.
\end{equation}

Now if we consider the Grassmann bundle of $(n-q_{2})$-planes in
$TM$, i.e.,$ \widetilde{\pi}_{n-q_{2}} : G_{n-q_{2}}(TM)
\longrightarrow M $, then we have the exact sequence

$$ 0 \longrightarrow \widetilde{T}_{n - q_{2}}^{\nu} \longrightarrow \widetilde{\pi}_{n - q_{2}}^{\ast}
TM \longrightarrow \widetilde{N}_{n - q_{2}}^{\nu} \longrightarrow
0.$$

\begin{rmk} The fiber of the fibre bundle $N_{n - q_{2}}^{\nu} \longrightarrow
G_{n-q_{2}}(TM)$ over an $(n - q_{2})$-plane
$E_{n-q_{2}} \in G_{n-q_{2}}(TM)$ is the $q_{2}$-plane

$$ (\widetilde{N}_{n - q{2}}^{\nu})_{E_{n-q_{2}}} \simeq \frac{T_{x}M}{E_{n -
q_{2}}} \simeq E_{q_{2}} .$$

\end{rmk}

Le us consider  $ \widetilde{\pi}_{q_{1}} : G_{q_{1}}(\widetilde{N}_{n -
q_{2}}^{\nu}) \longrightarrow G_{n-q_{2}} (TM) $ be the Grassmann
bundle of $q_{1}$-planes in $\widetilde{N}_{n - q_{2}}^{\nu}$, we
have the flag bundle ${\widetilde{\pi}} : F_{q_{1}, q_{2}} (TM)
\longrightarrow M $ of $(q_{1},q_{2})$-planes in $TM$ where
$\widetilde{\pi} = \widetilde{\pi}_{n-q_{2}} \circ
\widetilde{\pi}_{q_{1}}.$

\begin{rmk} A point of $F_{q_{1}, q_{2}}(TM)$ over $x \in M$
means first an $q_{2}$-plane $E_{q_{2}}$ in $T_{x}M$ and then an
$q_{1}$-plane $E_{q_{1}}$ in $E_{q_{2}}$; this is a Flag in
$T_{x}M$. This construction was  motivated by  [\ref{IL}].

\end{rmk}

\begin{definition} We define the Nash modification of $M$ with respect of
the flag $\mathcal{F} = (\mathcal{F}_{1}, \mathcal{F}_{2})$ by
$$ M^{\nu} = \overline{Im \sigma}$$
where the closure is taken over the flag bundle $F_{q_{1},
q_{2}}(TM)$ and $\sigma$ is the natural section  of $\widetilde{\pi}$ induced by the flag
$\mathcal{F}.$

\end{definition}
If we consider the projections ${p}_{i} : F_{q_{1},q_{2}}(TM)
\longrightarrow G_{q_{i}}(TM); \ \ i = 1,2,$ then we can take the
pullback of exact sequences (\ref{eq.3.1}) and (\ref{eq.3.2}) to
$M^{\nu}.$
\begin{equation}\label{eq.3.3}
0 \longrightarrow p_{1}^{\ast}T_{1}^{\nu} \longrightarrow
p_{1}^{\ast}\pi_{1}^{\ast} TM \longrightarrow
p_{1}^{\ast}N_{1}^{\nu} \longrightarrow 0
\end{equation}
\begin{equation}\label{eq.3.4}
0 \longrightarrow p_{2}^{\ast}T_{2}^{\nu} \longrightarrow
p_{2}^{\ast}\pi_{2}^{\ast} TM \longrightarrow
p_{2}^{\ast}N_{2}^{\nu} \longrightarrow 0.
\end{equation}

One has the easy propositions:

\begin{proposition}\label{Prop.1.1.4} The following diagram

$$ \xymatrix{                               & M^{\nu}\ar[dd]^{\pi} \ar[rd]^{p_{2}} \ar[ld]_{p_{1}} &                                \\
             M_{1}^{\nu} \ar[rd]_{\pi_{1}}  &                                                      & M_{2}^{\nu} \ar[ld]^{\pi_{2}}  \\
                                            &         M                                            &                                 \\
}$$

\noi is commutative.

\end{proposition}

\begin{proposition}\label{Prop.1.1.5} On $M^{\nu}$ we have the exact sequences

\begin{equation}\label{eq.3.5}
0 \longrightarrow N_{12}^{\nu} \longrightarrow
p_{1}^{\ast}N_{1}^{\nu} \longrightarrow p_{2}^{\ast}N_{2}^{\nu}
\longrightarrow 0
\end{equation}

\begin{equation}\label{eq.3.6}
0 \longrightarrow p_{1}^{\ast}T_{1}^{\nu} \longrightarrow
p_{2}^{\ast}T_{2}^{\nu} \longrightarrow N_{12}^{\nu} \longrightarrow
0
\end{equation}

\noi where $N_{12}^{\nu} := p_{2}^{\ast}T_{2}^{\nu} /
p_{1}^{\ast}T_{1}^{\nu}.$

\end{proposition}

The  Proposition \ref{Prop.1.1.4} and Proposition \ref{Prop.1.1.5}  imply  that $p_{1}^{\ast}\pi_{1}^{\ast} TM =
p_{2}^{\ast}\pi_{2}^{\ast} TM = \pi^{\ast}TM$. Therefore, we have the following diagram on $M^{\nu}$.
$$ \xymatrix{ 0 \ar[rd]  &  & 0 \ar[ld] &  & 0  \\
   & p_{1}^{\ast}T_{1}^{\nu} \ar[rd] \ar[dd] &  & p_{2}^{\ast}N_{2}^{\nu} \ar[lu] \ar[ru] &  \\
   &  & \pi^{\ast}TM \ar[ru] \ar[rd]  &  &  \\
   & p_{2}^{\ast}T_{2}^{\nu} \ar[ru]   \ar[rd] &  & p_{1}^{\ast}N_{1}^{\nu} \ar[uu]  \ar[rd] &  \\
 0 \ar[ru] &  & N_{12}^{\nu} \ar[ru] \ar[rd]  &  & 0 \\
  & 0 \ar[ru] &   &  0 &  }$$
We define the normal bundle $N^{\nu}$
over $M^{\nu}$ by $N_{12}^{\nu} \oplus p_{2}^{\ast}N_{2}^{\nu}$ and also we define

$$ \varphi(N^{\nu}) := \varphi_{1} (N_{12}^{\nu}) \smile
\varphi_{2}(p_{2}^{\ast} N_{2}^{\nu}),$$

\noi where $\varphi_{i}$ is a homogeneous symmetric polynomial of degree $d_{i}$.

Let $S$ be a compact connected component of $S(\mathcal{F})$ and let $S^{\nu} = \pi^{-1}(S)$. Also, let $U^{\nu}$ be a neighborhood of $S^{\nu}$
in $M^{\nu}$ disjoint from the other components of $S(\mathcal{F})^{\nu}$. Let $\widetilde{U}_{1}^{\nu}$ be a regular neighborhood of $S^{\nu}$ in $F_{q_{1},q_{2}}(TM)$ with $\widetilde{U}_{1}^{\nu} \cap M^{\nu} \subset U^{\nu}$ and $\widetilde{U}_{0}^{\nu}$ a tubular neighborhood of $U_{0}^{\nu} = U^{\nu} \setminus S^{\nu}$ in $F_{q_{1},q_{2}}(TM)$.
We consider the covering $\mathcal{\widetilde{U}}^{\nu} = \{ \widetilde{U}_{0}^{\nu}, \widetilde{U}_{1}^{\nu} \}$ of $\widetilde{U}^{\nu} = \widetilde{U}_{0}^{\nu} \cup \widetilde{U}_{1}^{\nu}$. The characteristic class $\varphi(N^{\nu})$ is represented by the cocycle
$$\varphi(^{12}_{2}\nabla_{\ast}^{\nu}) =
\varphi_{1}(^{12}\nabla_{\ast}^{\nu}) \smile
\varphi_{2}(^{2}\nabla_{\ast}^{\nu}) \in
A^{2(d_{1}+d_{2})}(\widetilde{\mathcal{U}}^{\nu}),$$
 where
$$\varphi_{1}(^{12}\nabla_{\ast}^{\nu}) =
(\varphi_{1}(^{12}\nabla_{0}^{\nu}),
\varphi_{1}(^{12}\nabla_{1}^{\nu}),
\varphi_{1}(^{12}\nabla_{0}^{\nu},^{12}\nabla_{1}^{\nu}))$$

\noi and $\varphi_{2}(^{2}\nabla_{\ast}^{\nu}) =
(\varphi_{2}(^{2}\nabla_{0}^{\nu}),
\varphi_{2}(^{2}\nabla_{1}^{\nu}),
\varphi_{2}(^{2}\nabla_{0}^{\nu},^{2}\nabla_{1}^{\nu})   ).$ \\

Here $^{12}\nabla_{0}^{\nu} \ \mbox{and} \ ^{12}\nabla_{1}^{\nu}$ are
connections on the bundle $N_{12}^{\nu}$ over $\widetilde{U}_{0}^{\nu} \ \ \mbox{and} \ \
\widetilde{U}_{1}^{\nu}$ respectively, and $^{2}\nabla_{0}^{\nu}, ^{2}\nabla_{1}^{\nu}$ are connections on the bundle 
$p_{2}^{\ast}N_{2}^{\nu}$ over $\widetilde{U}_{0}^{\nu} \ \ \mbox{and} \ \
\widetilde{U}_{1}^{\nu}$ respectively.  \\

If we set $U = \pi(U^{\nu})$, which is  a neighborhood of $S$ on $M$, then $\pi$ induces a biholomorphic map $U_{0}^{\nu} \longrightarrow U_{0} = U\setminus S $. Over  $U_{0}$ we have the  basic (Bott) connections $\nabla_{12}$ and
$\nabla_{2}$ on $N_{12}$ and $N_{F_{2}^{0}}$ respectively. Let us consider  $^{12}\nabla_{0}^{\nu}$ the connection for $N_{12}^{\nu}$ given by  $^{12}\nabla_{0}^{\nu} = \pi^{\ast}( \nabla_{12})$ and analogously $^{2}\nabla_{0}^{\nu}= \pi^{\ast}( \nabla_{2})$ for $p_{2}^{\ast}N_{2}^{\nu}$ then
$$ \varphi(^{12}_{2}\nabla_{0}^{\nu}) =
\varphi_{1}(^{12}\nabla_{0}^{\nu}) \smile
\varphi_{2}(^{2}\nabla_{0}^{\nu}) = 0. $$
The cocycle  $ \varphi(^{12}_{2}\nabla_{\ast}^{\nu}) \in A^{2(d_{1}+d_{2})} (\widetilde{\mathcal{U}}^{\nu},
\widetilde{U}_{0}^{\nu})$ defines a class 
$$\varphi_{S^{\nu}} (N^{\nu};\mathcal{F}) \in H^{2(d_{1}+d_{2})}
(U^{\nu}, U^{\nu}\setminus S^{\nu}; \mathbb{C}  ).$$ Its image by Alexander isomorphism 
$H^{2(d_{1}+d_{2})}
(U^{\nu}, U^{\nu}\setminus S^{\nu}; \mathbb{C}  ) \to H_{2n -
2(d_{1} + d_{2})} (S^{\nu} ; \mathbb{C})$ is denoted by
Res$_{\varphi_{1},\varphi_{2}}(N^{\nu},\mathcal{F},S^{\nu})$.

\begin{definition} The above class Res$_{\varphi_{1},\varphi_{2}}(N^{\nu},\mathcal{F},S^{\nu})$ is called  the Nash residue  at $S^{\nu}$ of the flag $\mathcal{F}$ with respect to $\varphi = (\varphi_{1},\varphi_{2})$.
\end{definition}

\subsection{Comparison of Baum-Bott and Nash residues for flags}

In this section, we compare the previously defined Nash residue for flags with the Baum-Bott residue for flags.  The result is
generalization of the one obtained in [\ref{BS}]. This comparison implies a partial answer to the rationality conjecture for flags.

Let $M$ be a complex manifold of dimension $n$ and $\mathcal{F} =
(\mathcal{F}_{1}, \mathcal{F}_{2})$ be a 2-flag of singular
holomorphic foliations of dimensions $(q_{1}, q_{2})$ on $M$. Also
let $S $ be a compact connected component of $S(\F)$ and denote 
$S^{\nu} = \pi^{-1}(S)$ as above. Then there is a canonical
homomorphism
$$ \pi_{\ast} : H_{2n - 2d}(S^{\nu}; \mathbb{C}) \longrightarrow H_{2n - 2d}(S;
\mathbb{C}).$$

\begin{thm} Let $\varphi = (\varphi_{1}, \varphi_{2})$ be homogeneous symmetric polynomials of degree
$d_{1}$ and $d_{2}$ respectively, satisfying the condition $(1.1)$ of the Bott vanishing theorem for flags. If
$\varphi_{i}$ has integral coefficients, then the difference
$$\mbox{Res}_{\varphi_{1},\varphi_{2}}(\mathcal{N}_{\mathcal{F}},\mathcal{F}, S) -  \pi_{\ast}\mbox{Res}_{\varphi_{1},\varphi_{2}}(N^{\nu},\mathcal{F},
S^{\nu}) $$
 is in the image of the canonical homomorphism $H_{2n - 2d}(S;
\mathbb{Z}) \longrightarrow H_{2n - 2d}(S; \mathbb{C})$, i.e., is an integral class.

\end{thm}

\begin{proof}  Let us consider  analytic resolutions of the
sheaves $\mathcal{F}_{1}$ and $\mathcal{F}_{2}$
$$ 0 \longrightarrow A_{U} (E_{q}^{12}) \longrightarrow \cdots \longrightarrow A_{U} (E_{1}^{12})
 \longrightarrow A_{U} \otimes \mathcal{F}_{1} \longrightarrow 0$$
$$ 0 \longrightarrow A_{U} (E_{r}^{2}) \longrightarrow \cdots \longrightarrow A_{U} (E_{1}^{2})
 \longrightarrow A_{U} \otimes \mathcal{F}_{2} \longrightarrow 0.$$
 The exact sequences

$$ \xymatrix{
           &                         & 0 \ar[d]                       &                          &      \\
 0 \ar[r]  & \mathcal{F}_{1} \ar[r]  & \mathcal{F}_{2} \ar[r] \ar[d]  & \mathcal{N}_{1,2} \ar[r]  & 0  \\
           &                         & \Theta_{M} \ar[d]              &                          &  \\
           &                         & \mathcal{N}_{2} \ar[d]         &                          &  \\
           &                         & 0                              &                          &  }$$

\noi provide  resolutions of the sheaves $\mathcal{N}_{1,2} \ \
\mbox{and} \ \ \mathcal{N}_{2}:$

$$\displaystyle 0 \longrightarrow A_{U} (E_{q}^{12}) \stackrel{\eta_{q}^{12}}{\longrightarrow} \cdots \longrightarrow A_{U} (E_{1}^{1})
\stackrel{\eta_{1}^{12}}{\longrightarrow}A_{U} (F_{2}^{0}) \longrightarrow   A_{U} \otimes \mathcal{N}_{1,2} \longrightarrow 0$$

$$\displaystyle 0 \longrightarrow A_{U} (E_{q}^{2}) \stackrel{\eta_{r}^{2}}{\longrightarrow} \cdots \longrightarrow A_{U} (E_{1}^{2})
\stackrel{\eta_{1}^{2}}{\longrightarrow}A_{U} (TM) \longrightarrow   A_{U} \otimes \mathcal{N}_{2} \longrightarrow 0.$$
and then  exact sequences of vector bundles on $U_{0}$:

\begin{equation}\label{eq.3.7}
0 \longrightarrow E_{q}^{12} \longrightarrow \cdots \longrightarrow
E_{1}^{12} \longrightarrow F_{2}^{0} \longrightarrow N_{1,2}
\longrightarrow 0
\end{equation}

\begin{equation}\label{eq.3.8}
0 \longrightarrow E_{r}^{2} \longrightarrow \cdots \longrightarrow
 E_{1}^{2} \longrightarrow TM \longrightarrow N_{F_{2}^{0}}
\longrightarrow 0.
\end{equation}

The sheaves homomorphisms $\eta_{j}^{12}$ and $\eta_{i}^{2}$
induce bundle homomorphisms on $U$ and $U^{\nu}$

$$ h_{j}^{12} : E_{j}^{12} \longrightarrow E_{j-1}^{12} $$

$$ h_{i}^{2} : E_{i}^{2} \longrightarrow E_{i-1}^{2} $$

$$ \pi^{\ast}h_{j}^{12} : \pi^{\ast}E_{j}^{12} \longrightarrow \pi^{\ast}E_{j-1}^{12} $$

$$ \pi^{\ast}h_{i}^{2} : \pi^{\ast}E_{i}^{2} \longrightarrow \pi^{\ast}E_{i-1}^{2}. $$

We claim that

\begin{equation}\label{eq. 19}
Im(\pi^{\ast}h_{1}^{2}) \subset
p_{2}^{\ast}T_{2}^{\nu} \ \ \mbox{and} \ \ Im(\pi^{\ast}h_{1}^{12}) \subset
p_{1}^{\ast}T_{1}^{\nu} \ \ \mbox{on} \ \ U^{\nu}.
\end{equation}

In fact, away from the singular sets we have exact 
sequences

$$ 0 \longrightarrow \mathcal{F}_{1} \longrightarrow \mathcal{F}_{2}
\longrightarrow \mathcal{N}_{1,2}\longrightarrow 0 $$

$$ 0 \longrightarrow p^{\ast}_{1}T_{1}^{\nu} \longrightarrow p^{\ast}_{2}T_{2}^{\nu}
\longrightarrow N_{1,2}^{\nu}\longrightarrow 0. $$

Note that $T_{1}^{\nu} = \pi_{1}^{\ast}F_{1}^{0}$ on $M_{1}^{\nu}$
 implies that we have  $p^{\ast}_{1}T_{1}^{\nu} = \pi^{\ast}F_{1}^{0}$.
Analogously we have $p^{\ast}_{2}T_{2}^{\nu} = \pi^{\ast}F_{2}^{0}.$ \\

Then, we obtain  the exact sequences
$$ \xymatrix{   &             &       \vdots \ar[d]       &                &                 \\
           &                                 & \pi^{\ast}E_{1}^{12} \ar^{\pi^{\ast}h_{1}^{12}}[d]               &                                &      \\
 0 \ar[r]  & \pi^{\ast}F_{1}^{0} \ar[r]  & \pi^{\ast}F_{2}^{0} \ar[r]                     & \pi^{\ast}N_{1,2} \ar[r]  &   0.   }$$

Therefore, away from the singular sets, which are dense in
$U^{\nu}$, we have the inclusions  in (\ref{eq. 19}). So by   continuity arguments we
have the inclusions of (\ref{eq. 19}) in $U^{\nu}$. Then we have two complexes of vector bundles on $U^{\nu}$,
which are exact on $U_{0}^{\nu}.$

\begin{equation}\label{eq.3.9}
0 \longrightarrow \pi^{\ast}(E_{q}^{12}) \longrightarrow ... \longrightarrow
\pi^{\ast}(E_{1}^{12}) \longrightarrow \pi^{\ast}F_{2}^{0}
\longrightarrow N_{1,2}^{\nu} \longrightarrow 0
\end{equation}

\begin{equation}\label{eq.3.10}
0 \longrightarrow \pi^{\ast}(E_{r}^{2}) \longrightarrow ... \longrightarrow
\pi^{\ast}(E_{1}^{2}) \longrightarrow \pi^{\ast}TM \longrightarrow
p_{2}^{\ast}N_{2}^{\nu} \longrightarrow 0.
\end{equation}

Let us denote  the virtual bundles $\widetilde{\varepsilon}_{12}
= \widetilde{\pi}^{\ast}(\xi^{12}) - N_{12}^{\nu}$ \ \ and \ \
$\widetilde{\varepsilon}_{2} = \widetilde{\pi}^{\ast}(\xi^{2}) -
p_{2}^{\ast}N_{2}^{\nu}$.

By classical properties of characteristic classes , we can  write
\begin{equation}\label{eq.3.11}
\varphi_{1}( \widetilde{\pi}^{\ast}(\xi^{12})) = \varphi_{1}(
N_{12}^{\nu}) + \sum \varphi_{1}^{i}(N_{12}^{\nu})
\psi_{1}^{i}(\widetilde{\varepsilon}_{12}),
\end{equation}
where the $\varphi_{1}^{i}$ are symmetric polynomials  with
integral coefficients and $\psi_{1}^{i}$ are symmetric polynomials
with integral coefficients without constant term. Analogously

\begin{equation}\label{eq.3.12}
\varphi_{2}( \widetilde{\pi}^{\ast}(\xi^{2})) = \varphi_{2}(
p_{2}^{\ast}N_{2}^{\nu}) + \sum
\varphi_{2}^{i}(p_{2}^{\ast}N_{2}^{\nu})
\psi_{2}^{i}(\widetilde{\varepsilon}_{2}).
\end{equation}
By taking the cap product of $(\ref{eq.3.11})$ with $(\ref{eq.3.12})$ we have \\

\noi $ \varphi_{1}( \widetilde{\pi}^{\ast}\xi^{12}) . \varphi_{2}(
\widetilde{\pi}^{\ast}\xi^{2}) = \varphi_{1}( N_{12}^{\nu}). \varphi_{2}(p_{2}^{\ast}
N_{2}^{\nu}) + \varphi_{1}( N_{1,2}^{\nu}). \sum
\varphi_{2}^{i}(p_{2}^{\ast}N_{2}^{\nu})
\psi_{2}^{i}(\widetilde{\varepsilon}_{2})  + $  \\

\noi $+ \sum \varphi_{1}^{i}(N_{1,2}^{\nu})
\psi_{1}^{i}(\widetilde{\varepsilon}_{1}) . \varphi_{2}(
p_{2}^{\ast}N_{2}^{\nu}) + \sum \varphi_{1}^{i}(N_{12}^{\nu})
\psi_{1}^{i}(\widetilde{\varepsilon}_{12}) .
\varphi_{2}^{i}(p_{2}^{\ast}N_{2}^{\nu})
\psi_{2}^{i}(\widetilde{\varepsilon}_{2}).$   \\

\noi oi $H^{2(d_{1} + d_{2})}(U^{\nu}).$  \\

\noi We claim that we have a ``good localization", i.e., in
$A^{\ast}(\widetilde{U}^{\nu}, \widetilde{U}^{\nu}_{0})$ we have \\

\noi $ \varphi(
\widetilde{\pi}^{\ast}(^{12}_{2}\nabla_{\ast}^{\bullet}) ) =
\varphi_{1}( \widetilde{\pi}^{\ast}(^{12}\nabla_{\ast}^{\bullet})) .
\varphi_{2}( \widetilde{\pi}^{\ast}(^{2}\nabla_{\ast}^{\bullet})) =
$    \\

\noi $= \varphi_{1}(^{12}\nabla^{\nu}_{\ast}). \varphi_{2}(
^{2}\nabla^{\nu}_{\ast}) + \varphi_{1}(^{12}\nabla^{\nu}_{\ast}).
\sum \varphi_{2}^{i}(^{2}\nabla^{\nu}_{\ast})
\psi_{2}^{i}(^{2}\nabla^{\varepsilon}_{\ast})  + $  \\

\noi $+ \sum \varphi_{1}^{i}(^{12}\nabla^{\nu}_{\ast})
\psi_{1}^{i}(^{12}\nabla^{\varepsilon}_{\ast}) . \varphi_{2}(
^{2}\nabla^{\nu}_{\ast}) + \sum
\varphi_{1}^{i}(^{12}\nabla^{\nu}_{\ast})
\psi_{1}^{i}(^{12}\nabla^{\varepsilon}_{\ast}) .
\varphi_{2}^{i}(^{2}\nabla^{\nu}_{\ast})
\psi_{2}^{i}(^{2}\nabla^{\varepsilon }_{\ast}) + D\tau,$ \\

\noi where $\tau = (0,0, \tau_{01})$   with $\tau_{01} = \varphi_{1}(^{12}\nabla_{0}^{\nu}).
^{2}\tau_{01} + ^{12}\tau_{01} .\varphi_{2}(^{2}\nabla_{1}^{\nu}) +
^{12}\tau_{01} .\sum \varphi_{2}^{i}(^{2}\nabla_{1}^{\nu}).
\psi_{2}^{i}(^{2}\nabla_{1}^{\varepsilon})$  \\

\noi For further details on $^{2}\tau_{01}$ and $^{12}\tau_{01}$, we refer to
[\ref{BS}].  \\

\noi The above claim shows that, in
$H^{2(d_{1}+d_{2})}(U^{\nu},U^{\nu} \backslash S^{\nu} ,\mathbb{C}
)$  we have 

\noi$\pi^{\ast} \varphi_{S}(\mathcal{N}_{\mathcal{F}}, \mathcal{F},)
= \varphi_{S^{\nu}}(N^{\nu},\mathcal{F}) + \sum
\varphi_{1}(N_{1,2}^{\nu}). \varphi_{2}^{i}(p_{2}^{\ast}N_{2}^{\nu}).
\psi_{2,S}^{i}(\varepsilon_{2}) + $  \\

\noi $+ \sum \varphi_{1}^{i}(N_{1,2}^{\nu}).
\psi_{1,S}^{i}(\varepsilon_{12}).
\varphi_{2}(p_{2}^{\ast}N_{2}^{\nu}) + \sum
\varphi_{1}^{i}(N_{1,2}^{\nu}). \psi_{1,S}^{i}(\varepsilon_{12}).
\varphi_{2}^{i}(p_{2}^{\ast}N_{2}^{\nu}).
\psi_{2,S}^{i}(\varepsilon_{2}) .$   \\

\noi Then by the commutative diagram (where $A$ denotes the respective  Alexander isomorphisms)
$$ \xymatrix{
  H^{2(d_{1}+ d_{2})}(U^{\nu}, U^{\nu} \setminus S^{\nu}, \mathbb{C} )  \ar[d]^{A}  & H^{2(d_{1}+ d_{2})}(U, U \setminus S, \mathbb{C} ) \ar[l]_{\pi^\ast} \ar[d]^{A}   \\
  H_{2n -2(d_{1}+ d_{2})}(S^{\nu}, \mathbb{C} )   \ar[r]^{\pi^\ast}                        & H_{2n -2(d_{1}+ d_{2})}(S, \mathbb{C} )                          }$$

We obtain that the difference between these residues in $ H_{2n
-2(d_{1}+ d_{2})}(S, \mathbb{C} )$ is an integral classes.

\end{proof}

\begin{corollary} Let us consider the symmetric  polynomials   $\varphi_{1} = c_{i_{1}} \cdots c_{i_{r}}$ and $\varphi_{2} =
c_{j_{1}}\cdots c_{j_{t}}$ with $i_{\nu} > codim (\mathcal{F}_{1})$ for
some $\nu \in [1, \dots,r]$ or $i_{s} > codim (\mathcal{F}_{2})$ for some
$s \in [1,\dots ,t]$, then the Baum-Bott residue Res$_{\varphi_{1},\varphi_{2}}(\mathcal{N}_{\mathcal{F}},
\mathcal{F}, S )$ for the flag $ \mathcal{F} = (\mathcal{F}_{1}, \mathcal{F}_{2} )$ is an integral class.

\end{corollary}

\end{document}